\documentclass{amsart}

\usepackage{mathtools}
\DeclarePairedDelimiter{\ceil}{\lceil}{\rceil}

\usepackage{amsmath}
\usepackage{amssymb} 
\usepackage{amsthm}  
\usepackage[pdftex,  colorlinks=true]{hyperref} 

\usepackage{tikz-cd} 

\newtheorem{thm}{Theorem}[section]
\newtheorem{theorem}[thm]{Theorem} \newtheorem{proposition}[thm]{Proposition}  
\newtheorem{lemma}[thm]{Lemma}

\newtheorem*{claim}{Claim}

\newtheorem{thmx}{Theorem}
\newtheorem{corx}[thmx]{Corollary}
\newtheorem{quex}[thmx]{Question}

\theoremstyle{definition}

\newtheorem{definition}[thm]{Definition}

\newtheorem{remark}[thm]{Remark}

\newtheorem{example}[thm]{Example}

\numberwithin{equation}{section}

\DeclareMathOperator{\dist}{\mathsf{dist}}

\newcommand{\ZZ}{\mathbb{Z}}

\usepackage{tikz}
\usetikzlibrary{arrows,quotes}
\tikzstyle{blackNode}=[fill=black, draw=black, shape=circle]

\DeclareMathOperator{\wcop}{\mathsf{wCop}}
\DeclareMathOperator{\scop}{\mathsf{sCop}}

\newcommand{\sq}[2]{#1 ~\square ~ #2}
\newcommand{\lx}[2]{#1 \times #2}
\newcommand{\xsq}[2]{#1 \boxtimes  #2}
\newcommand{\cir}[3]{#1 \circ_{#3}  #2}

\title{ Coarse geometry  of the Cops and robber game}

\author[Lee]{Jonathan Lee}
\author[Mart\'inez-Pedroza]{Eduardo Mart\'inez Pedroza}
\author[Rodr\'iguez-Quinche]{Juan Felipe Rodr\'iguez-Quinche}

\email{eduardo.martinez@mun.ca}
\email{frodriguezqu@mun.ca}

\address{Department of Mathematics and Statistics. Memorial University of Newfoundland. St. John's, NL. Canada}

\date{\today}

\begin{document}
\maketitle
\begin{abstract}
We introduce two variations of the cops and robber game on graphs. These games yield two invariants in $\ZZ_+\cup\{\infty\}$ for any connected graph $\Gamma$, the \emph{weak cop number $\wcop(\Gamma)$} and \emph{strong cop number $\scop(\Gamma)$}. These invariants satisfy that $\scop(\Gamma)\leq\wcop(\Gamma)$. Any graph that is finite or a tree has strong cop number one. These new invariants are preserved under small local perturbations of the graph, specifically, both the weak and strong cop numbers are quasi-isometric invariants of connected graphs. More generally, we prove that if $\Delta$ is a quasi-retract of $\Gamma$ then $\wcop(\Delta)\leq\wcop(\Gamma)$ and $\scop(\Delta)\leq\scop(\Gamma)$. We exhibit families of examples of graphs with arbitrary weak cop numbers (resp. strong cop number). We prove that hyperbolic graphs have strong cop number one. We also prove that one-ended non-amenable locally-finite vertex-transitive graphs have infinite weak cop number. We raise the question of whether there exists a connected vertex transitive graph with finite weak (resp. strong) cop number different than one.
\end{abstract}
\section{Introduction}
\begin{figure}[t]
\centering
\begin{tikzpicture}[scale=0.8]
\node at (0,-0.3) {$u$};
\node at (2,-2) {\large$\mathbf{\Gamma_2}$};
\foreach \Point in {(0,0),(4,0),(8,0),(6,0.5),(6,-0.5),(8,1),(8,-1)}
\filldraw[black] \Point circle (2pt);
\foreach \Point in {(9,2.25),(9,-2.25)}
\draw (0,0) -- \Point;
\foreach \Point in {(11,2.75),(11,-2.75)}
\draw[dashed] (0,0) -- \Point;
\foreach \X in {1,2,3,4}{
\draw (2*\X,0)--++(0,0.5*\X);
\draw (2*\X,0)--++(0,-0.5*\X);
\filldraw[black] (2*\X,0.5*\X) circle (2pt);
\filldraw[black] (2*\X,-0.5*\X) circle (2pt);
}
\draw[dashed] (10,2.5)--(10,-2.5);
\end{tikzpicture}
\caption{The graph $\Gamma_2$ satisfies $\wcop(\Gamma_2)=\scop(\Gamma_2)=2$.}
\label{fig:two}
\end{figure}
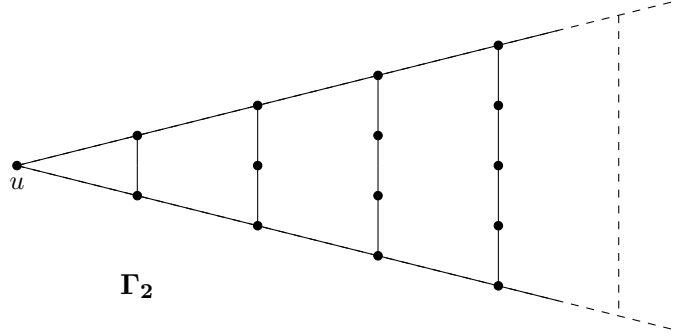

The cops and robber game was introduced independently by Quilliot~\cite{Quilliot} and Nowakowski and Winkler~\cite{Nowa}. This is a perfect information two player game on an undirected graph, where one player controls a set of cops and the other one controls a single robber. On the graph each cop and the robber chooses a vertex to occupy, with the cops choosing first. The game then alternates between cops and the robber moving along adjacent vertices, with the cops moving first. The cops win if, after a finite number
of rounds, a cop occupies the same vertex as the robber; and the robber wins if he can avoid capture indefinitely. The cop number of a graph is the minimum number of cops necessary to always capture a robber. Different variations of this game have been studied specially on finite graphs, see the book by Bonato and Nowakowski~\cite{Bonato2011}.
In this article, we introduce a variation of the cops and robber game which defines invariants of connected graphs, the \emph{weak cop number} and the \emph{strong cop number}, which take values in $\mathbb{Z}_+\cup\{\infty\}$. The main motivation is to introduce invariants of infinite connected graphs that are robust under small local perturbations, or in the language of coarse geometry, quasi-isometric invariants. The variation of the game that we introduce is a combination of the one studied in~\cite{article1} where the cops and the robber have different speeds, and the one in ~\cite{article2, article3} where the cops capture if the distance to the robber is less than a fixed number.

Let us introduce some notation and conventions before describing our variation of the game. All graphs considered in this article are undirected, have no double edges, and have no loops.  The vertex and edge sets of a graph $\Gamma$ are denoted by $V(\Gamma)$ and $E(\Gamma)$  respectively.
The \emph{length} of a path in $\Gamma$ is the number of edges,  a path  of minimal length between  vertices $u$ and $v$ is called a \emph{geodesic} between $u$ to $v$, and the \emph{path distance} $\dist_\Gamma(u,v)$ is the length of a geodesic between $u$ and $v$. In a connected graph $\Gamma$, the \emph{path distance} $\dist_\Gamma$ defines a metric on $V(\Gamma)$. In  case  there is no ambiguity, we use $\dist$ instead of $\dist_\Gamma$.

 \subsection*{Description of the game.}
Consider a connected graph $\Gamma$ and let $\dist$ denote the \emph{path distance} on $V(\Gamma)$.  The variation of the cops and robber game on $\Gamma$ that we consider in this article is defined as follows. This is a perfect information game in which the players are a set of cops and a robber, and the game depends on the following parameters: the number of cops $n$, the speed of the cops $\sigma$, the speed of the robber $\psi$, the reach (or radius of capture) of the cops $\rho$, a vertex $v$ of $\Gamma$, and a positive integer $R$. Initially, the $n$ cops choose vertices $c^1_0,\ldots ,c^n_0$ as their initial positions. Then, knowing the initial positions of the cops, the robber chooses a vertex $r_0$ as his initial position. Then the cops and the robber move alternately with the cops moving first. A move of the cops followed by a move of the robber is called a \emph{stage}. The vertices where the cops and robber are located at the end of the $k$-stage are denoted as $c^1_k,\ldots ,c^n_k$ and $r_k$, respectively.
At the beginning of the $k$-stage, each cop can move from its current position to any vertex at a distance at most $\sigma$, that is,
$\dist(c^i_{k-1}, c^i_{k})\leq\sigma $ for every $i\in\{1,\ldots ,n\}$. The robber is \emph{captured} during the $k$-stage if $\dist(r_{k-1}, c^i_{k})\leq \rho$ for some $i\in\{1,\ldots ,n\}$. After the cops have moved, if the robber has not been captured, the robber can move from its current position $r_{k-1}$ to a vertex $r_k$ if there is a path from $r_{k-1}$ to $r_k$ of length at most $\psi$ such that every vertex in the path is at distance larger than $\rho$ from any cop.  The cops win the game if eventually they can protect the $R$-ball centered at $v$, that means, the robber is captured or there is $N>0$ such that $\dist(v,r_k)>R$ for all $k\geq N$.
\subsection*{Definition of the weak and strong cop numbers.}
We say that a graph $\Gamma$ is \emph{$\mathsf{CopWin}(n,\sigma,\rho,\psi, R)$} if for any vertex $v$ of $\Gamma$, $n$ cops with speed $\sigma$ and reach $\rho$ can eventually protect the $R$-ball centered at $v$.
The definitions of the weak cop number and the strong cop number differ in the order in which the parameters of the game are chosen by the two players. For the weak cop number, the robber has an information advantage, the cops choose their speed $\sigma$ and reach $\rho$ and then the robber, knowing this information, chooses his speed $\psi$ and the radius of the ball to protect $R$. For the strong cop number the cops have the advantage of choosing their reach $\rho$ after knowing the robber's speed. More precisely, these invariants are defined as follows:
\begin{itemize}
\item 	We say that $\Gamma$ is \emph{$n$-weak cop win} if there exists $\sigma \in \mathbb{Z}_{>0}$ and $\rho \in \mathbb{Z}_{\geq 0}$ such that for any $\psi, R \in \mathbb{Z}_{>0}$, $\Gamma$ is $\mathsf{CopWin}(n,\sigma,\rho,\psi , R)$. In symbols,
\[ \text{$\Gamma$ is $n$-weak cop win} \Longleftrightarrow \exists\ \sigma, \rho \ \forall\ \psi, R \colon \text{$\Gamma$ is } \mathsf{CopWin}(n,\sigma,\rho,\psi , R). \]
\item
	
We say that $\Gamma$ is \emph{$n$-strong cop win} if there exists $\sigma \in \mathbb{Z}_{>0}$ such that for any $\psi\in \mathbb{Z}_{>0}$, there is $\rho \in \mathbb{Z}_{\geq 0}$ such that for any $R \in \mathbb{Z}_{>0}$, $\Gamma$ is $\mathsf{CopWin}(n,\sigma,\rho,\psi , R)$. In symbols, 	
\[\text{$\Gamma$ is $n$-strong cop win} \leq n \Longleftrightarrow \exists\ \sigma\ \forall\ \psi\ \exists\ \rho\ \forall\ R \colon \text{$\Gamma$ is } \mathsf{CopWin}(n,\sigma,\rho,\psi , R). \]
\end{itemize}
The \emph{strong cop number $\scop(\Gamma)$} of a graph $\Gamma$ is defined as the
smallest value of $n$ such that $\Gamma$ is $n$-strong cop win, with $\scop(\Gamma)=\infty$ if there is no such $n$.
The \emph{weak cop number $\wcop(\Gamma)$} is defined analogously as the smallest $n$ such that $\Gamma$ is $n$-weak cop win.
Observe that
\[ 0< \scop(\Gamma) \leq \wcop(\Gamma) .\]
The infinite path has weak and strong cop number one, a single cop can push the robber away from any ball. An example of a graph $\Gamma_2$
such that $\wcop(\Gamma_2)=\scop(\Gamma_2)=2$ is illustrated in Figure~\ref{fig:two}. In $\Gamma_2$, two cops starting from the vertex $u$ and moving simultaneously away from $u$ along the two main ``branches" will either push a robber away, or enclose the robber in one of the vertical paths followed by capture; and one can give an argument that one cop is never enough if the ball to protect is large enough. These ideas are generalized and stated as Theorem~\ref{thmx:Range} which states the existence of graphs with arbitrary weak and strong cop numbers.
For the rest of the introduction, $\Gamma$ denotes a connected graph.
\subsection*{Strong cop number of hyperbolic graphs}
It is an observation that finite graphs and trees
have strong cop number one. The notion of hyperbolic graph, introduced by Gromov~\cite{gromov}, is a generalization of the notion of tree. All finite graphs are hyperbolic. For a definition of hyperbolic graph and the proof of the following result see Section~\ref{sec:hyperbolic}.
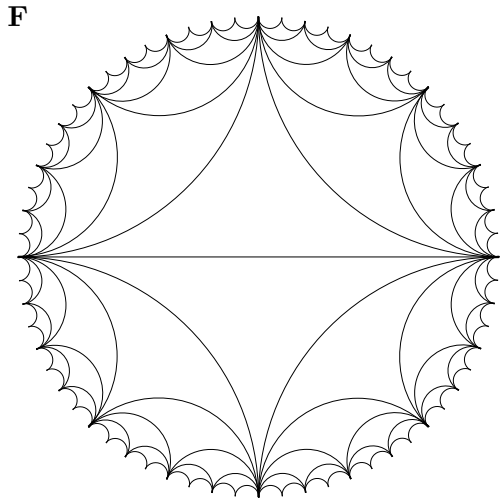
\begin{figure}[t]
\begin{tikzpicture}[scale=0.4]
\node at (-8,8) {\large $\mathbf{F}$};
\draw [ultra thin] (-8,0) -- (8,0);
\foreach \i in {0,1,2,3} {%
\draw [ultra thin] (90*\i:8) arc (270+90*\i:180+90*\i:8);}
\foreach \i in {0,1,...,7} {%
\draw [ultra thin] (45*\i:8) arc (270+45*\i:135+45*\i:3.3);}
\foreach \i in {0,1,...,15} {%
\draw [ultra thin] (22.5*\i:8) arc (270+22.5*\i:112.5+22.5*\i:1.6);}
\foreach \i in {0,1,...,31} {%
\draw [ultra thin] (11.25*\i:8) arc (270+11.25*\i:101.25+11.25*\i:0.8);}
\foreach \i in {0,1,...,63} {%
\draw [ultra thin] (5.625*\i:8) arc (270+5.625*\i:95.625+5.625*\i:0.4);}
\end{tikzpicture}
\caption{For the Farey graph $\mathbf F$, $\wcop(\mathbf F)=\scop(\mathbf F)=1$.}
\label{fig:Farey}
\end{figure}

\begin{figure}
\centering
\includegraphics[width=0.55\textwidth]{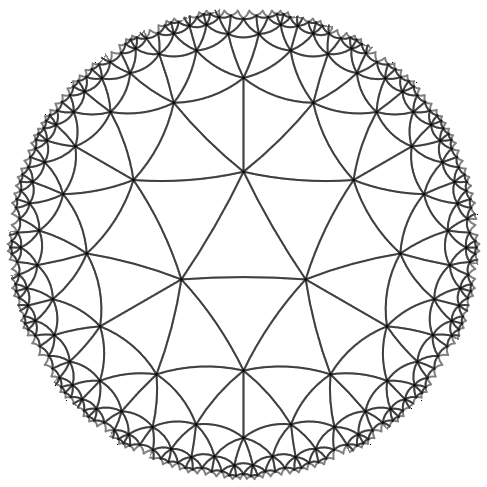}
\caption{Any graph $H$ arising as a regular tiling of the hyperbolic plane satisfies $\wcop(H)=\infty$ and $\scop(H)=1$. }
\label{fig:7-3-tiling}
\end{figure}

\begin{thmx}[Theorem~\ref{hypcopwin1}]\label{thmx:hyperbolic}
If $\Gamma$ is hyperbolic, then $\scop(\Gamma)=1$.
\end{thmx}
The class of hyperbolic graphs contains many interesting subclasses. For example, the \emph{Farey graph} is a hyperbolic graph that has as vertices all rational numbers expressed as reduced fractions $\frac pq$ with $q>0$, together with a vertex for $\frac 1 0$. Two vertices $\frac ab$ and $\frac cd$ span an edge when the matrix $\begin{pmatrix} a & c\\ b & d\end{pmatrix}$ has determinant $\pm 1$, an schematic of the graph is in Figure~\ref{fig:Farey} and for background see~\cite{HatcherTopNumbers}. It is an observation that a single cop with speed and reach one can protect balls of arbitrary radius in the Farey graph, and hence the weak and strong cop numbers of the Farey graph are equal to one.
The class of graphs arising as regular tilings of the hyperbolic plane are hyperbolic, for instance Figure~\ref{fig:7-3-tiling} shows the graph arising from the order-7 triangular tiling.
\begin{corx}
Any graph arising as a regular tiling of the hyperbolic plane has strong cop number one.
\end{corx}
In contrast to Theorem~\ref{thmx:hyperbolic}, the weak cop number of an arbitrary hyperbolic graph is not necessarily one.  For example, graphs in the  class defined by the above corollary have infinite weak cop number, see Corollary~\ref{corx:HyperbolicTilingWeakCop}. Let us mention here that the connection between cops and robber games and Gromov's hyperbolicity has been a subject of interest, see for example in~\cite{chapolin, article1}.
The graph arising as the regular tiling by squares of the Euclidean plane, known as the square grid, is a classical example of an infinite graph that is not hyperbolic.

\begin{thmx}[Theorem \ref{thm:SquareGrid}]
The square grid has infinite weak cop number.
\end{thmx}
Surprisingly, we have not been able to compute the strong cop number of the square grid. It is not difficult to verify that its strong cop number is larger than one. We suspect a positive answer to the following question.
\begin{quex}
Is the strong cop number of the square grid infinite?
\end{quex}
\begin{figure}
	\centering
\begin{tikzpicture}[scale=0.5]
\node at (-4,10) {\large $P^2$};
\foreach \Point in {(0,0),(2,0),(4,0),(6,0),(8,0),(10,0)}{
\draw[line width=0.2mm] \Point --++ (0,9);
\draw[line width=0.2mm] \Point --++ (0,-1);
\filldraw[black] \Point circle (0.1cm);
\filldraw[black] \Point --++ (0,2) circle (0.1cm);
\filldraw[black] \Point --++ (0,4) circle (0.1cm);
\filldraw[black] \Point --++ (0,6) circle (0.1cm);
\filldraw[black] \Point --++ (0,8) circle (0.1cm);
}
\foreach \Point in {(0,-1.8),(2,-1.8),(4,-1.8),(6,-1.8),(8,-1.8),(10,-1.8),
(0,9.8),(2,9.8),(4,9.8),(6,9.8),(8,9.8),(10,9.8)}{
\node [rotate=90] at \Point {\Large{$\ldots$}};
}
\foreach \Point in {(-1.8,0),(-1.8,2),(-1.8,4),(-1.8,6),(-1.8,8),
(11.8,0),(11.8,2),(11.8,4),(11.8,6),(11.8,8)}{
\node at \Point {\LARGE{$\ldots$}};
}
\foreach \Point in {(0,0),(0,2),(0,4),(0,6),(0,8)}{
\draw[line width=0.2mm] \Point --++ (11,0);
\draw[line width=0.2mm] \Point --++ (-1,0);
}
\end{tikzpicture}
	\caption{For the square grid, $\wcop(P^2)=\infty$ and $\scop(P^2)$ is unknown.}
	\label{fig:square}
\end{figure}
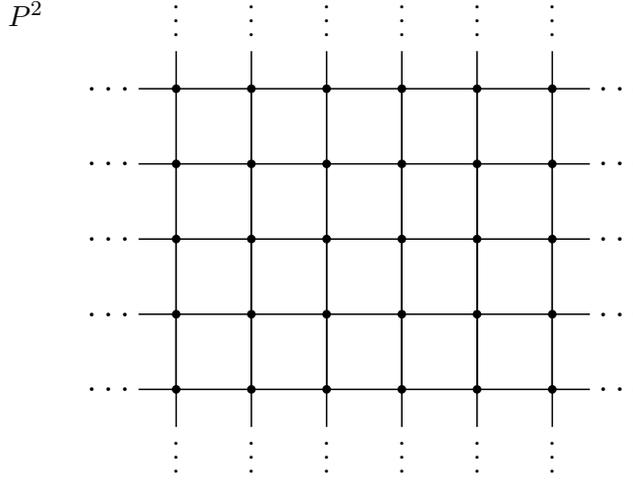
\subsection*{Cop numbers as large scale invariants}
The weak and strong cop numbers behave well under  quasi-retractions~\cite{Alonso}.
The notion of quasi-retraction is a generalization of the classical notion of graph retraction. A quasi-retraction does not necessarily target a subgraph. Briefly, a quasi-retraction from a graph $\Gamma$ to a graph $\Delta$ is a pair $(f,g)$ where $f\colon \Gamma\to \Delta$ and $g\colon\Delta\to\Gamma$ are Lipschitz maps (with respect to the path distances) between the sets of vertices such that $g\circ f$ is uniformly closed to the identity function on the vertex set of $\Delta$. In Section~\ref{sec:QuasiRetractions}, we provide a precise definition of quasi-retraction and prove the following result.
\begin{thmx}[Theorem~\ref{corocopN1}]\label{thm:THEO1}
Let $\Gamma$ and $\Delta$ be connected graphs. If $\Delta$ is a quasi-retract of $\Gamma$, then
\[\wcop(\Delta) \leq \wcop(\Gamma) \quad \text{and} \quad \scop(\Delta) \leq \scop(\Gamma).\]
\end{thmx}
For graphs $\Gamma$ and $\Delta$, let
$\sq{\Gamma}{\Delta}$, $\xsq{\Gamma}{\Delta}$, $ \lx{\Gamma}{\Delta}$, $\cir{\Gamma}{\Delta}{y} $ denote their Cartesian product, strong product, lexicographic product, and rooted product (over a vertex $y$ of $\Delta$), respectively. In all cases, $\Gamma$ is a graph retract of the any of these products, hence:
\begin{corx}\label{CnProd}
Let $\Gamma$ and $\Delta$ be connected graphs, and let $y\in V(\Delta)$. Suppose
\[\Lambda\in\{\sq{\Gamma}{\Delta}, \xsq{\Gamma}{\Delta}, \lx{\Gamma}{\Delta}, \cir{\Gamma}{\Delta}{y}\}.\] Then
\[\wcop(\Gamma) \leq \wcop(\Lambda) \quad \text{and} \quad \scop(\Gamma) \leq \scop(\Lambda).\]
\end{corx}
The main consequence of Theorem~\ref{thm:THEO1} is the quasi-isometric invariance of the weak and strong cop numbers. The connected graphs $\Gamma$ and $\Delta$ are \emph{quasi-isometric} if there is a pair of maps $f\colon\Delta\to\Gamma$ and $g\colon\Gamma\to \Delta$ such that $(f,g)$ and $(g,f)$ are quasi-retractions from $\Gamma$ to $\Delta$ and from $\Delta$ to $\Gamma$, respectively.
\begin{corx}
If $\Gamma$ and $\Delta$ are connected quasi-isometric graphs, then
\[\wcop(\Delta) = \wcop(\Gamma) \quad \text{and} \quad \scop(\Delta) = \scop(\Gamma).\]
\end{corx}
There are similar results to this corollary for other combinatorial games. For example, it is shown in~\cite{firefighter} that for a variation of Hartnell’s firefighter game~\cite{Hartnell}, the existence of a winning strategy for the firefighters is a quasi-isometric invariant. For another variation of the firefighter game, a similar result was obtained in~\cite{MaPr23}.

\subsection*{Weak cop numbers of non-amenable one-ended graphs}
Definitions of the terminology used in the following result, as well as its proof, are the contents of Section~\ref{sectionamenable}. Let us briefly describe some of the terminology in the theorem.
A graph is \emph{vertex transitive} $\Gamma$ if for any two vertices $u$ and $v$, there is some automorphism of $\Gamma$ that maps $u$ to $v$.  A connected graph $\Gamma$ is \emph{one-ended} if for any finite subset of vertices the induced subgraph $\Gamma\setminus K$ has only one unbounded connected component. An example of a one-ended graph is the infinite square grid.
Intuitively, a connected graph $\Gamma$ is \emph{non-amenable} if it does not have large bottlenecks.
More formally, for a subset of vertices $K$ of $\Gamma$, let $\partial K$ be the set of edges of $\Gamma$ with one endpoint in $K$ and the other endpoint not in $K$. The graph $\Gamma$ is \emph{non-amenable} if its Cheeger constant $h(\Gamma)$ is nonzero, that is \[0<h(\Gamma)=\inf\left\{ \frac{|\partial K|}{|K|} \mid {K\subset V(\Gamma),\ |K|<\infty} \right\} \]

\begin{thmx}[Theorem~\ref{1eN-A}]
If $\Gamma$ is a connected, one-ended, non-amenable, locally finite, vertex transitive graph, then $\wcop(\Gamma)=\infty$.
\end{thmx}
While this result looks technical due to the number of hypotheses, the class of graphs where the theorem applies is large as the following corollaries illustrate.
All graphs arising as regular tilings of the hyperbolic plane are quasi-isometric, one-ended and non-amenable. In particular, all of them have the same weak and strong cop number.
\begin{corx}
Every graph arising as a regular tiling of the hyperbolic plane has infinite weak cop number.
\end{corx}
 
A standard source of vertex transitive locally finite graphs are Cayley graphs of finitely generated groups. Specifically, if $G$ is a group generated by a finite set $S$, the Cayley graph $\Gamma(G,S)$ is the graph with vertex set $G$ and and edge set all 2-subsets of the form $\{g, gs\}$ with $g\in G$ and $s\in S$. It is well known that $\Gamma(G,S)$ is a vertex transitive connected graph, and any two Cayley graphs of a finitely generated group with respect to finite generating sets are quasi-isometric~\cite{brid}. In particular, the quasi-isometric invariance of the weak and strong cop number implies that they induce invariants of finitely generated groups. Specifically for a finitely generated group $G$, let
\[ \wcop(G) = \wcop(\Gamma(G,S)) \qquad \scop(G) = \scop(\Gamma(G,S))\]
where $S$ is any finite generating set of $G$.
In the class of locally finite graphs, the properties of being one-ended, being hyperbolic and being non-amenable are quasi-isometric invariants of graphs, see~\cite{brid} and~\cite[Section 18]{DK18} for non-amenability. A finitely generated group is said to be hyperbolic, one ended or non-amenable depending on its Cayley graphs. One-ended hyperbolic groups are non-amenable~\cite{DK18}.
\begin{corx} \label{corx:HyperbolicTilingWeakCop}
Let $G$ is a finitely generated one-ended hyperbolic group, then $\wcop(G)=\infty$ and $\scop(G)=1$.
\end{corx}
In the sense of some probabilistic constructions introduced by Gromov, most finitely generated groups are hyperbolic~\cite{Gromov93} and one-ended~\cite{DGP11}. It is natural to ask whether the weak and strong cop numbers are meaningful invariants of finitely generated groups:
\begin{quex}
Is there a finitely generated group $G$ such that $1<\scop(G)<\infty$ or $1<\wcop(G)<\infty$?
\end{quex}

It is a remarkable result of Eskin, Fisher and White~\cite{EsFiWh12} (based on graphs discovered by Diestel and Leader~\cite{DiLe01})  that there exist  connected, locally finite, vertex transitive graphs that are not quasi-isometric to the Cayley graph of a finitely generated group.  

\begin{quex}
Is there a connected vertex transitive locally finite graph $\Gamma$ such that $1<\scop(\Gamma)<\infty$ or $1<\wcop(\Gamma)<\infty$?
\end{quex}

\subsection*{Range of the weak and strong cop numbers.}
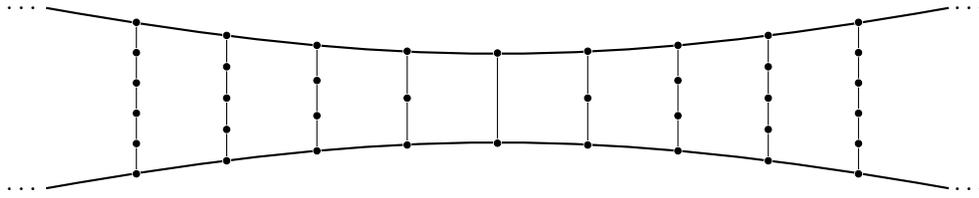
\begin{figure}
	\centering
\begin{tikzpicture}[scale=0.6]
\draw[black,thick] (-10,4) to [out=350,in=190] (10,4);
\draw[black,thick] (-10,0) to [out=10,in=170] (10,0);
\draw (0,3) -- (0,1);
\draw (-2,0.96)--(-2,3.04);
\draw (2,0.96)--(2,3.04);
\draw (-4,0.83)--(-4,3.17);
\draw (4,0.83)--(4,3.17);
\draw (-6,0.61)--(-6,3.39);
\draw (6,0.61)--(6,3.39);
\draw (-8,0.32)--(-8,3.68);
\draw (8,0.32)--(8,3.68);
\foreach \Point in {(0,3), (0,1), (-2,0.96),(-2,3.04),(2,0.96),(2,3.04),
(-4,0.83),(-4,3.17),(4,0.83),(4,3.17),
(-6,0.61),(-6,3.39),(6,0.61),(6,3.39),
(-8,0.32),(-8,3.68),(8,0.32),(8,3.68),(2,2), (-2,2),(4,1.61), (4,2.39), (-4,1.61) ,(-4,2.39),(-6,1.305), (-6,2), (-6,2.695),(6,1.305), (6,2), (6,2.695),(8,0.992), (8,1.664), (8,2.336), (8,3.008), (-8,0.992) ,(-8,1.664), (-8,2.336), (-8,3.008)
}{
\filldraw[white] \Point circle (2.5pt);
}
\foreach \Point in { (0,3), (0,1), (-2,0.96),(-2,3.04),(2,0.96),(2,3.04),
(-4,0.83),(-4,3.17),(4,0.83),(4,3.17),
(-6,0.61),(-6,3.39),(6,0.61),(6,3.39),
(-8,0.32),(-8,3.68),(8,0.32),(8,3.68),(2,2), (-2,2),(4,1.61), (4,2.39), (-4,1.61) ,(-4,2.39),(-6,1.305), (-6,2), (-6,2.695),(6,1.305), (6,2), (6,2.695),(8,0.992), (8,1.664), (8,2.336), (8,3.008), (-8,0.992) ,(-8,1.664), (-8,2.336), (-8,3.008)
}{
\filldraw \Point circle (2pt);
}
		\node at (-10.5,4) {$\dots$};
		\node at (10.5,4) {$\dots$};
		\node at (10.5,0) {$\dots$};
		\node at (-10.5,0) {$\dots$};
		\end{tikzpicture}
	\caption{The $\Theta_2$-extension of the infinite path $P$.}
	\label{fig:Theta2Z}
\end{figure}
\begin{figure}
	\centering
\begin{tikzpicture}[scale=0.6]
\draw (-1.6,3.2)--(-1.4,2.2)--(-1.8,1)--cycle;
\draw (-0.3,3.15)--(-0.5,0.93)--(0,2.1)--cycle;
\draw (-2.6,2.3)--(-3.07,1.1)--(-3,3.3)--cycle;
\draw (-4,2.45)--(-4.7,3.55)--(-4.7,1.1)--cycle;
\draw (-5.5,2.65)--(-6.4,3.87)--(-6.5,1.07)--cycle;
\draw (-7.3,2.88)--(-8.4,4.27)--(-8.8,1.01)--cycle;
\draw (1.6,2.04)--(1,3.18)--(0.8,0.8)--cycle;
\draw (3.2,2)--(2.4,3.27)--(2.2,0.63)--cycle;
\draw (5.6,2)--(4.2,3.425)--(3.8,0.43)--cycle;
\draw (8.2,2)--(6.8,3.7)--(5.8,0.17)--cycle;
\foreach \Point in {(-1.6,3.2),(-0.3,3.15),(-3,3.3),(-4.7,3.55),(-1.4,2.2), (0,2.1),(-2.6,2.3),(-4,2.45),(-1.8,1),(-0.5,0.93),(-3.07,1.1),(-4.7,1.1),(-5.5,2.65),(-6.4,3.87),(-6.5,1.07),(-7.3,2.88),(-8.4,4.27),(-8.8,1.01),(1.6,2.04),(1,3.18),(0.8,0.8),(3.2,2),(2.4,3.27),(2.2,0.63),(5.6,2),(4.2,3.425),(3.8,0.43),(8.2,2),(6.8,3.7),(5.8,0.17),(6.3,2),(-0.4,2),(-0.2,1.6),(-0.13,2.55),
(-2.75,1.9),(-3.04,2.2),(-2.75,2.7),(-4.23,2),(-4.45,1.6),(-4.17,2.73),(-4.38,3.04),(-4.7,2.65),(-4.7,2),(-6.1,3.45),(-5.9,3.2),(-5.72,2.95),(-5.79,2.2),(-6,1.85),(-6.22,1.5),(-6.425,3),(-6.45,2.4),(-6.47,1.9),(-8.1,3.9),(-7.9,3.65),(-7.7,3.4),(-7.5,3.155),(-7.6,2.5),(-7.9,2.125),(-8.2,1.75),(-8.5,1.4),(-8.7,1.8),(-8.63,2.35),(-8.57,2.9),(-8.5,3.5),(0.93,2.4),(0.87,1.6),(1.32,1.6),(1.1,1.25),(1.4,2.4),(1.2,2.8),(2.3,1.95),(2.35,2.55),(2.25,1.35),(2.47,1),(2.9,1.6),(2.7,1.3),(3,2.3),(2.83,2.6),(2.64,2.9),(4.14,3),(4.06,2.4),(3.95,1.5),(3.88,1),(4.5,3.1),(4.75,2.85),(5,2.6),(5.25,2.35),(5.155,1.6),(4.8,1.3),(4.45,1),(4.1,0.7),(6.3,1.9),(6.15,1.4),(5.98,0.8),(6.455,2.5),(6.6,3),(7.1,3.35),(7.5,2.85),(7.3,3.1),(7.7,2.6),(7.95,2.3),(6.25,0.5),(7,1.085),(6.625,0.7925),(7.8,1.7),(7.4,1.3925),(-1.67,2.2),(-0.4,2.1),(-3.05,2.34),(-4.7,2.53),(-6.43,2.77),(0.9,2.05),(2.3,2.02),(4,2.025)
}{
\filldraw[white] \Point circle (2pt);
}
\filldraw[white] (-8.54,3.045) circle (2.5pt);
\filldraw[white] (-8.55,3.035) circle (2.5pt);
\node[rotate = -15] at (-10,4.7) {$\dots$};
\node[rotate = 5] at (-10,0.9) {$\dots$};
\node[rotate = -5] at (-10,3.2){$\dots$};
\node[rotate = 0] at (9.7,2) {$\dots$};
\node[rotate = -10] at (7.5,-0.1) {$\dots$};
\node[rotate = 15] at (10,4.1) {$\dots$};
\foreach \Point in { (-1.6,3.2),(-0.3,3.15),(-3,3.3),(-4.7,3.55),(-6.4,3.87),(-8.4,4.27),(1,3.18),(2.4,3.27),(4.2,3.425),(6.8,3.7), (-1.4,2.2), (0,2.1),(-2.6,2.3),(-4,2.45),(-5.5,2.65),(-7.3,2.88),(1.6,2.04),(3.2,2),(5.6,2),(8.2,2),(-1.8,1),(-0.5,0.93),(-3.07,1.1),(-4.7,1.1),(-6.5,1.07),(-8.8,1.01),(0.8,0.8),(2.2,0.63),(3.8,0.43),(5.8,0.17)
}{
\filldraw \Point circle (1.5pt);
}
\foreach \Point in {(-0.4,2),(-0.2,1.6),(-0.13,2.55),
(-2.75,1.9),(-3.04,2.2),(-2.75,2.7),(-4.23,2),(-4.45,1.6),(-4.17,2.73),(-4.38,3.04),(-4.7,2.65),(-4.7,2),(-6.1,3.45),(-5.9,3.2),(-5.72,2.95),(-5.79,2.2),(-6,1.85),(-6.22,1.5),(-6.425,3),(-6.45,2.4),(-6.47,1.9),(-8.1,3.9),(-7.9,3.65),(-7.7,3.4),(-7.5,3.155),(-7.6,2.5),(-7.9,2.125),(-8.2,1.75),(-8.5,1.4),(-8.7,1.8),(-8.63,2.35),(-8.57,2.9),(-8.5,3.5),(0.93,2.4),(0.87,1.6),(1.32,1.6),(1.1,1.25),(1.4,2.4),(1.2,2.8),(2.3,1.95),(2.35,2.55),(2.25,1.35),(2.47,1),(2.9,1.6),(2.7,1.3),(3,2.3),(2.83,2.6),(2.64,2.9),(4.14,3),(4.06,2.4),(3.95,1.5),(3.88,1),(4.5,3.1),(4.75,2.85),(5,2.6),(5.25,2.35),(5.155,1.6),(4.8,1.3),(4.45,1),(4.1,0.7),(6.3,1.9),(6.15,1.4),(5.98,0.8),(6.455,2.5),(6.6,3),(7.1,3.35),(7.5,2.85),(7.3,3.1),(7.7,2.6),(7.95,2.3),(6.25,0.5),(7,1.085),(6.625,0.7925),(7.8,1.7),(7.4,1.3925)
}{
\filldraw \Point circle (1.5pt);
}
\draw (-9.5,1) .. controls (-2,1.2) .. (7,0);
\draw (-9.5,3.2) .. controls (-1,2) .. (9.2,2);
\draw (-9.5,4.5) .. controls (-1.2,2.8) .. (9.5,4);
		\end{tikzpicture}
		\caption{$\Theta_3$-extension of the infinite path}
		\label{fig:t3}
	\end{figure}
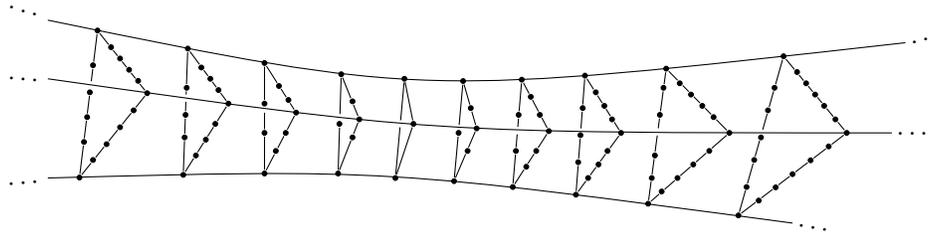
In Section~\ref{sec:ThetaExtensions} we describe some constructions that yield graphs with arbitrary weak and strong cop numbers. Let us summarize the results.
Let $\Gamma$ be a graph and let $n$ be a positive integer $n$.
We define a graph that we call the $\Theta_n$-extension $\Theta_n(\Gamma)$ of $\Gamma$. The graph $\Theta_n(\Gamma)$ quasi-retracts to $\Gamma$ and if $d$ is an upper bound for the degree of vertices of $\Gamma$, then $d+n$ is an upper bound for the degree of vertices of $\Theta_n(\Gamma)$.
As an example, the graphs in Figures~\ref{fig:Theta2Z} and ~\ref{fig:t3} illustrate the $\Theta_2$-extension and the $\Theta_3$-extension of the infinite path, respectively. The graph $\Theta_n(\Gamma)$ quasi-retracts to $\Gamma$ and allows us to produce a variety of examples of graphs with different weak and strong cop numbers.
\begin{thmx}[Corollary~\ref{TcopN}]
For any connected graph $\Gamma$ and for any integer $n>0$,
\[ \wcop(\Gamma) \leq \wcop(\Theta_n(\Gamma)) \leq n\cdot\wcop(\Gamma)\] and
\[\scop(\Gamma) \leq \scop(\Theta_n(\Gamma))\leq n\cdot\scop(\Gamma).\]
\end{thmx}
A graph is unbounded if it has infinite diameter. For example, infinite locally finite connected graphs are unbounded by Konig's lemma.
\begin{thmx}[Theorem~\ref{CoroTcop}]\label{thmx:Range}
Let $\Gamma$ be a connected unbounded graph.
If $\wcop(\Gamma)=1$, then $\wcop(\Theta_n(\Gamma))=n$.
Analogously, if $\scop(\Gamma)=1$, then $\scop(\Theta_n(\Gamma))=n$.
\end{thmx}
The $\Theta_n$-extensions allow us to construct graphs that have different weak and strong cop numbers. Specifically,
using trees and graphs arising as a regular tiling of the hyperbolic plane we obtain:
\begin{corx}
Let $n$ be a positive integer.
\begin{enumerate}
\item There is a graph $\Delta$ such that $\wcop(\Delta)=n$ and $\scop(\Delta) = n$.
\item There is a graph $\Gamma$ such that $\wcop(\Gamma)=\infty$ and $\scop(\Gamma)=n$.
\end{enumerate}
\end{corx}
It is not difficult to verify that for any pair of connected graphs $\Gamma$ and $\Delta$ that
\[ \wcop(\Gamma \vee \Delta) = \max\{\wcop(\Gamma),\wcop(\Delta)\},\qquad \scop(\Gamma \vee \Delta) = \max\{\scop(\Gamma),\scop(\Delta)\}\]
where $\Gamma \vee \Delta$ is the wedge sum, that is, the graph obtained by identifying a vertex of $\Gamma$ with a vertex of $\Delta$. More generally, if $\{\Delta_i\mid i\in I\}$ is a collection of graphs,
\[ \wcop\left(\bigvee_{i\in I} \Delta_i\right) = \sup\{\wcop\left(\Delta_i\right)\mid i\in I\},\qquad \scop\left(\bigvee_{i\in I} \Delta_i\right) = \sup\{\scop(\Delta_i)\mid i\in I\}.\]
Using the graphs provided by the previous corollary and using an infinite wedge sum one can produce a locally finite connected graph with infinite strong cop number. We do not have examples addressing the following question.
\begin{quex}\label{quex:Range}
Let $n,m$ arbitrary positive integers, such that $n>m$. Does there exist a graph $\Gamma$ with $\wcop(\Gamma)=n$ and $\scop(\Gamma)=m$?
\end{quex}
\subsection*{Acknowledgements}
The authors thank the referees for their comments,   suggestions and corrections. The article was substantially improved due to their detailed reports. Most results in this article are part of the Master projects of the first and third authors under the supervision of the second author. 
The authors also thank Danny Dyer and Stephen Finbow for comments and corrections on the Master theses. The revision of this article was done during a visit of the second author to Tokyo Metropolitan University; the second author thanks Professor Tomohiro Fukaya  for his hospitality. The second author also acknowledges funding by the Natural Sciences and Engineering Research Council of Canada, NSERC.

\section{The square grid has infinite weak cop number}
\begin{theorem}\label{thm:SquareGrid}
	The square grid has infinite weak cop number.
\end{theorem}
\begin{proof}	
	Consider a game on the infinite square grid such that there are $n$ cops each with speed $\sigma$ and reach $\rho$.
	We prove below that the cops cannot protect any ball of radius $R=2(n + n\rho + \rho +\sigma)(n+3)$ from a robber with speed $\psi=2R$.
	
	In the grid, consider a rectangle with base of length
	$2(n + n\rho + \rho +\sigma)(n+2)$ and height
	$2(n + n\rho + \rho +\sigma)$ contained in the ball
	${B_R}(u_0)$. The set of vertices on the boundary or in the interior of the rectangle can be partitioned into $n+2$ squares, each square with side of length $2(n + n\rho + \rho +\sigma)$. From this point forward, we will refer to these squares, all vertices and edges on and within their respective boundaries, simply as $\emph{squares}$ and refer to a square with no cops inside as a $\emph{cop-free square}$.
	
	There are $n+2$ squares and $n$ cops, so at any point in the game there must be at least $2$ cop-free squares. Assume that at the beginning of the game, the robber positions himself at the center of a cop-free square. Note that a robber at the center of a cop-free square is at a distance greater than $n + n\rho + \rho +\sigma$ from any cop.
	
	During each of the following stages of the game, the robber remains at the center of their cop-free square or, if their square is no longer cop-free, moves to one that is. We need to show that such a move is possible:
	
	\begin{lemma}[Robber's move during a single stage]
		If at the beginning of the stage the robber is located at the center of a square which is cop-free. Then, after the cops have moved in this stage, there is a path from the location of the robber to the center of a cop-free square that does not pass within reach of a cop and has length $\psi$ or less.
	\end{lemma}
	\begin{proof}
		At the beginning of the stage, all cops are at  distance greater than $n + n\rho + \rho +\sigma$ from the robber, so after the cops  have moved, this distance must still be greater than $n + n\rho + \rho$. The robber can therefore move a distance of $n + n\rho$ in any direction and remain outside the reach of the cops.
		
		Denote the geodesic from the robber's location to the center of a chosen cop-free square by $p_0$. Note that this is a horizontal path. Define the paths $p_i$ to be the vertical translations of $p_0$ by $i$-units, where $-(n + n\rho) < i \leq n + n\rho$.
		
		The maximum number of paths from $\{p_i\}$ that contain a vertex that a fixed cop can reach is $2\rho+1$, so the maximum number of paths that $n$ cops can reach is $2\rho n+n$. There are a total of $2(n + n\rho)$ paths in $\{p_i\}$, so there are at least $2(n + n\rho)-(2\rho n+n)=n$ paths from $\{p_i\}$ that are out of reach of the cops.
		
		Note that the starting point of each $p_i$ is distance $n + n\rho$ or less from the robber's location and the endpoint of each is distance $n + n\rho$ or less from the center of a cop-free square. The endpoint of each $p_i$ is thus out of reach of all cops and there is a path from the endpoint to the center of the cop-free square that is also out of reach. It follows that if $p_i$ is out of reach of all cops, then the concatenated path from the robber's location to $p_i$ to the center of the cop-free square is also out of reach of all cops. The length of this path is less than or equal to $2(n + n\rho) + 2(n + n\rho+\rho+\sigma)(n+1) < 2R = \psi$. Hence the robber can reach the center of a cop-free square without being captured during this stage.
		\end{proof}
		
		This defines a winning strategy for the robber in a game with $n$ cops with arbitrary speed $\sigma$ and arbitrary reach $\rho$ on the square grid.
\end{proof}
\section{Hyperbolic Graphs are $1$-strong cop win}\label{sec:hyperbolic}
There are several equivalent definitions of $\delta$-hyperbolic space, we refer the reader to~\cite{brid} for a survey. For our purposes, we use the following:
\begin{definition}\cite[Chapter III.H, Definition 1.1]{brid}
Let $X$ be a geodesic metric space and let $\delta>0$. A triangle is said to be \emph{$\delta$-slim} if each one of its sides is contained in the $\delta$-neighborhood of the union of the other two sides, see Figure~\ref{fig:deltaslim}.
A geodesic metric space $X$ is said to be \textit{$\delta$-hyperbolic} if every geodesic triangle is $\delta$-slim. A metric space is \emph{hyperbolic} if it is $\delta$-hyperbolic for some $\delta\geq0$.
\end{definition}
\begin{figure}[ht]
\centering
\begin{tikzpicture}[scale=1.1]
\draw[black!50,line width=1pt,dashed] (1.25,-2.85) -- (1.95,-3.55);
\node at (1.3,-3.45){$\delta$};
\draw[black!50,line width=1pt,dashed] (4.7,-2.7) -- (4.25,-3.6);
\node at (4.725,-3.4){$\delta$};
\draw[black!50,line width=1pt,dashed] (4.39,-4) -- (4.255,-5.05);
\node at (4.5725,-4.525){$\delta$};
\draw[line width=1pt]
(22.89429pt, -149.875pt) .. controls (11.09529pt, -149.875pt) and (-0.093216pt, -142.8693pt) .. (-4.860962pt, -131.281pt)
-- (-4.860962pt, -131.281pt)
-- (-4.860962pt, -131.281pt)
-- (-4.860962pt, -131.281pt) .. controls (-11.16472pt, -115.9585pt) and (-3.853714pt, -98.42651pt) .. (11.46878pt, -92.1235pt)
-- (11.46878pt, -92.1235pt)
-- (11.46878pt, -92.1235pt)
-- (11.46878pt, -92.1235pt) .. controls (40.46753pt, -80.19327pt) and (45.47528pt, -76.70874pt) .. (52.38504pt, -23.737pt)
-- (52.38504pt, -23.737pt)
-- (52.38504pt, -23.737pt)
-- (52.38504pt, -23.737pt) .. controls (54.52779pt, -7.307495pt) and (69.58253pt, 4.275513pt) .. (86.01279pt, 2.131256pt)
-- (86.01279pt, 2.131256pt)
-- (86.01279pt, 2.131256pt)
-- (86.01279pt, 2.131256pt) .. controls (102.4423pt, -0.011505pt) and (114.0238pt, -15.06775pt) .. (111.881pt, -31.49649pt)
-- (111.881pt, -31.49649pt)
-- (111.881pt, -31.49649pt)
-- (111.881pt, -31.49649pt) .. controls (107.6915pt, -63.616pt) and (103.0356pt, -85.82574pt) .. (91.6333pt, -104.3073pt)
-- (91.6333pt, -104.3073pt)
-- (91.6333pt, -104.3073pt)
-- (91.6333pt, -104.3073pt) .. controls (77.30679pt, -127.5273pt) and (55.61005pt, -138.8417pt) .. (34.2973pt, -147.61pt)
-- (34.2973pt, -147.61pt)
-- (34.2973pt, -147.61pt)
-- (34.2973pt, -147.61pt) .. controls (30.56305pt, -149.1467pt) and (26.69604pt, -149.875pt) .. (22.8943pt, -149.875pt) -- cycle
;
\draw[line width=1pt]
(187.865pt, -158.1242pt) .. controls (183.701pt, -158.1242pt) and (179.474pt, -157.2534pt) .. (175.4383pt, -155.4137pt)
-- (175.4383pt, -155.4137pt)
-- (175.4383pt, -155.4137pt)
-- (175.4383pt, -155.4137pt) .. controls (170.5033pt, -153.2079pt) and (146.396pt, -146.5164pt) .. (116.0735pt, -143.0717pt)
-- (116.0735pt, -143.0717pt)
-- (116.0735pt, -143.0717pt)
-- (116.0735pt, -143.0717pt) .. controls (81.6575pt, -139.1612pt) and (52.29202pt, -140.9004pt) .. (33.38977pt, -147.9677pt)
-- (33.38977pt, -147.9677pt)
-- (33.38977pt, -147.9677pt)
-- (33.38977pt, -147.9677pt) .. controls (17.86852pt, -153.7697pt) and (0.585526pt, -145.8917pt) .. (-5.216476pt, -130.3727pt)
-- (-5.216476pt, -130.3727pt)
-- (-5.216476pt, -130.3727pt)
-- (-5.216476pt, -130.3727pt) .. controls (-11.01848pt, -114.8529pt) and (-3.140472pt, -97.56842pt) .. (12.37852pt, -91.76642pt)
-- (12.37852pt, -91.76642pt)
-- (12.37852pt, -91.76642pt)
-- (12.37852pt, -91.76642pt) .. controls (76.95203pt, -67.6254pt) and (177.287pt, -90.31442pt) .. (200.33pt, -100.8204pt)
-- (200.33pt, -100.8204pt)
-- (200.33pt, -100.8204pt)
-- (200.33pt, -100.8204pt) .. controls (215.4066pt, -107.6942pt) and (222.0545pt, -125.4872pt) .. (215.1808pt, -140.5629pt)
-- (215.1808pt, -140.5629pt)
-- (215.1808pt, -140.5629pt)
-- (215.1808pt, -140.5629pt) .. controls (210.146pt, -151.6021pt) and (199.2545pt, -158.1234pt) .. (187.865pt, -158.1241pt) -- cycle
;
\draw[line width=1pt]
(187.88pt, -158.125pt) .. controls (184.457pt, -158.125pt) and (180.9755pt, -157.5348pt) .. (177.5735pt, -156.2898pt)
-- (177.5735pt, -156.2898pt)
-- (177.5735pt, -156.2898pt)
-- (177.5735pt, -156.2898pt) .. controls (168.1303pt, -152.8345pt) and (159.2488pt, -149.8518pt) .. (150.6598pt, -146.9673pt)
-- (150.6598pt, -146.9673pt)
-- (150.6598pt, -146.9673pt)
-- (150.6598pt, -146.9673pt) .. controls (122.5273pt, -137.5195pt) and (98.23105pt, -129.3603pt) .. (80.0488pt, -112.288pt)
-- (80.0488pt, -112.288pt)
-- (80.0488pt, -112.288pt)
-- (80.0488pt, -112.288pt) .. controls (59.1808pt, -92.69424pt) and (50.32104pt, -65.341pt) .. (52.16605pt, -26.2045pt)
-- (52.16605pt, -26.2045pt)
-- (52.16605pt, -26.2045pt)
-- (52.16605pt, -26.2045pt) .. controls (52.94604pt, -9.653503pt) and (66.99579pt, 3.130249pt) .. (83.54529pt, 2.349487pt)
-- (83.54529pt, 2.349487pt)
-- (83.54529pt, 2.349487pt)
-- (83.54529pt, 2.349487pt) .. controls (100.0948pt, 1.569489pt) and (112.8793pt, -12.47952pt) .. (112.0993pt, -29.02975pt)
-- (112.0993pt, -29.02975pt)
-- (112.0993pt, -29.02975pt)
-- (112.0993pt, -29.02975pt) .. controls (110.1643pt, -70.07651pt) and (119.2633pt, -73.13126pt) .. (169.7608pt, -90.08951pt)
-- (169.7608pt, -90.08951pt)
-- (169.7608pt, -90.08951pt)
-- (169.7608pt, -90.08951pt) .. controls (178.3078pt, -92.95975pt) and (187.9933pt, -96.21249pt) .. (198.1918pt, -99.94452pt)
-- (198.1918pt, -99.94452pt)
-- (198.1918pt, -99.94452pt)
-- (198.1918pt, -99.94452pt) .. controls (213.7513pt, -105.6385pt) and (221.7485pt, -122.8675pt) .. (216.0545pt, -138.427pt)
-- (216.0545pt, -138.427pt)
-- (216.0545pt, -138.427pt)
-- (216.0545pt, -138.427pt) .. controls (211.607pt, -150.5838pt) and (200.1155pt, -158.125pt) .. (187.88pt, -158.125pt) -- cycle
;
\draw[red,line width=1pt]
(187.8831pt, -128.1167pt) .. controls (118.8831pt, -102.8672pt) and (78.6328pt, -101.8675pt) .. (82.13306pt, -27.6167pt)
;
\draw[red,line width=1pt]
(82.13306pt, -27.6167pt) .. controls (74.63306pt, -85.11771pt) and (66.63281pt, -101.8675pt) .. (22.88306pt, -119.8667pt)
;
\draw[red,line width=1pt]
(22.88306pt, -119.8667pt) .. controls (78.38306pt, -99.11795pt) and (170.8828pt, -120.3654pt) .. (187.8831pt, -128.1167pt)
;
\end{tikzpicture}
\caption{$\delta$-slim condition }
\label{fig:deltaslim}
\end{figure}
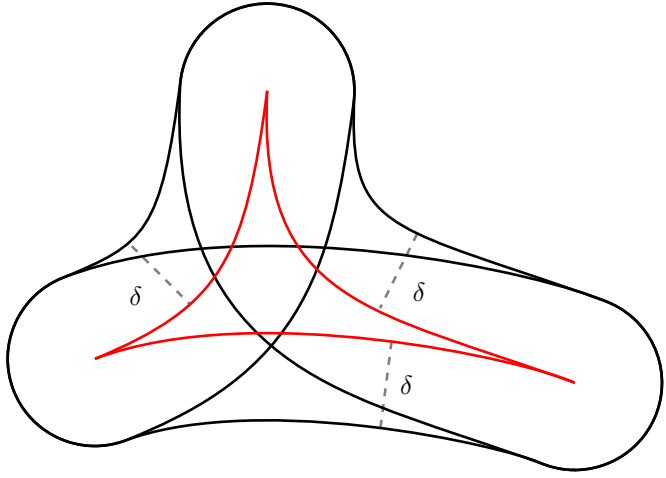
A connected graph can be considered a geodesic metric space by regarding each edge as a segment of length one and imposing the path metric, that is, the distance between any two points is the length of a shortest path.   A connected graph is hyperbolic if it is a hyperbolic metric space with respect to the path metric.

\begin{theorem}\label{hypcopwin1}
Hyperbolic graphs are $1$-strong cop win.
\end{theorem}
The proof of the theorem relies on the following result.
\begin{lemma}\label{lemmajf1}\cite[Chapter III.H, Proposition 1.6]{brid}
Let $X$ be a $\delta$-hyperbolic geodesic space, let $C$ be a continuous rectifiable path in $X$, if $[p,q]$ is a geodesic segment connecting the endpoints of $C$, then for all $x\in [p,q]$:
\[ \dist(x,C)\leq \delta |\log_2(l(C))|+1,\]
with $l(C)$ the length of $C$.
\end{lemma}
In a graph, an edge path is always a continuous rectifiable path with respect to the path metric.   

\begin{proof}[Proof of Theorem \ref{hypcopwin1}]
Let $\Gamma$ be a $\delta$-hyperbolic graph and, by enlarging $\delta$ if necessary, assume that $\delta$ is an integer. Consider the game on $\Gamma$ with a robber with arbitrary speed $\psi$ and a single cop with speed $\sigma=2\delta+1$ and reach $\rho=\delta |\log_2(\psi)|+\delta + \psi+1$. Observe that the reach of the cop depends on the robber's speed.
Let $u_0$ be an arbitrary vertex of $\Gamma$ and let $R>0$. We will prove that the cop can protect the closed ball ${B_R}(u_0)$. Assume the initial location of the cop is $c_0=u_0$ and the initial position of the robber is $r_0$ and let $g_1$ be a geodesic path from $u_0$ to $r_0$. At the beginning of the $(n+1)$-stage, the position of the cop is denoted by $c_{n}$ and the position of the robber is $r_{n}$. We describe a strategy for the cop that guarantees the following two conditions for every $n$:
\begin{enumerate}
\item \label{Condition01} If after the cop moves on the $(n+1)$-stage, the robber has not been captured; then $c_{n+1}$ is a vertex of a geodesic, denoted by  $g_n$, from $u_0$ to $r_{n}$,
\item \label{Condition02} $\dist(u_0,c_n) < \dist(u_0,c_{n+1})$. \end{enumerate}
\begin{figure}[ht]
\centering
\begin{tikzpicture}[scale=1.3]
\draw[black,thick] (0,1) to [out=-45,in=90] (2,-3);
\draw[black,thick] (4,1) to [out=225,in=90] (2,-3);
\draw[black,thick] (0,1) to [out=-45,in=225] (4,1);
\draw[thick,black!70] plot [smooth,tension=1.1] coordinates{(0,1) (0.4,1.5) (1,1.3) (1.5,1.7) (2,1.5) (2.7,1.1) (3.1,1.4) (3.5,1.5) (4,1)};
\foreach \X in {(0,1), (4,1) , (2,-3), (1.6,-1.1)}
\filldraw[black] \X circle (2pt);
\node at (-0.4,1) {$r_{n}$};
\node at (4.6,1) {$r_{n+1}$};
\node at (2.4,-3) {$u_0$};
\node at (2.7,-1.7) {$g_{n+1}$};
\node at (1.5,-1.7) {$g_{n}$};
\node at (2.2,0.4) {$g$};
\node at (2.5,1.5) {$C$};
\node at (1.1,-1) {$c_{n+1}$};
\end{tikzpicture}
\caption{Proof of Theorem~\ref{hypcopwin1} }
\label{fig:deltriangle}
\end{figure}
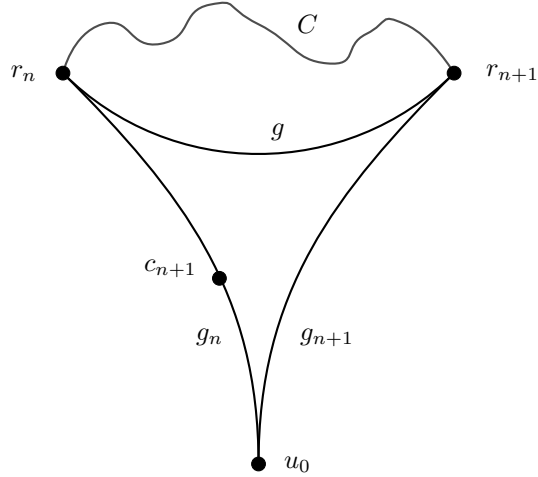
Let us recall that at the beginning of every stage, the cop moves first. The strategy for the cop is defined below:
\begin{itemize}
\item In the first stage, the cop moves in the direction of the robber along $g_1$.
\item If at the beginning of the $(n+2)$-stage the robber has not been captured, then the movement of the cop is described as follows.
Consider a $\delta$-slim triangle with vertices $r_{n}$, $r_{n+1}$ and $u_0$ and with sides $g_{n}$, $g_{n+1}$ and $g$, where $g$ is defined as a geodesic path between $r_{n}$ and $r_{n+1}$. We assume that $c_{n+1}$ is a vertex of the geodesic $g_n$. See Figure~\ref{fig:deltriangle}.
The assumption that the robber was not captured during the $(n+1)$-stage  implies that $\dist(c_{n+1},g)> \delta$. Indeed, suppose that $\dist(c_{n+1},g)\leq \delta$. Then there exists a vertex $x$ in the geodesic $g$ such that $\dist(c_{n+1},x) \leq \delta$. Suppose that the robber moved from $r_n$ to $r_{n+1}$ along the path $C$, and hence its length satisfies $l(C)\leq \psi$. Lemma \ref{lemmajf1} implies that \[\dist(c_{n+1},C)\leq \dist(c_{n+1},x)+\dist(x,C)\leq \delta +\delta |\log_2(l(C))|+1 \leq \rho, \]
and therefore the robber was captured. This contradicts our hypothesis and therefore $\dist(c_{n+1},g)>\delta$.
Since the triangle is $\delta$-slim, there is a vertex $y$ on $g_{n+1}$ at a distance less than or equal to $\delta$ from $c_{n+1}$. The cop will move first to $y$. If the robber has not been captured yet, then $\dist(y,r_{n+1})>\rho>\delta+1$ and the cop moves $\delta+1$ units along the geodesic $g_{n+1}$ in the direction of $r_{n+1}$. Note that the cop can do this total movement due to the assumption that $2\delta+1\leq \sigma$. Moreover,
\begin{align*}
\dist(u_0,c_{n+2}) &=
\dist(u_0,y) +\delta+1 \\
&\geq \dist (u_0,c_{n+1})- \dist(c_{n+1},y)+\delta+1\\
&\geq \dist (u_0,c_{n+1}) -\delta + \delta +1 \\
&> \dist (u_0,c_{n+1}).
\end{align*}
\end{itemize}
Let us verify that this is a winning strategy for the cop.
Assume the cop follows the strategy and the robber is never captured. After the cop moves, on the $(n+1)$-stage, $\dist(u_0,r_n)= \dist(u_0,c_{n+1})+\dist(c_{n+1},r_n)$ due to condition~\eqref{Condition01}. In particular $\dist(u_0,r_n)\geq \dist(u_0,c_{n+1})$. After a finite number of steps, let us say $m$, the cop will be located at a distance at least $R$ from $u_0$, this is $\dist(u_0,c_{n})\geq R$ for every $n\geq m$, due to condition~\eqref{Condition02}. It follows that $\dist(u_0,r_n)\geq R$ for any $n\geq m$, that means the cop has been able to protect the ball.
\end{proof}
\section{Cop numbers and quasi-retractions}\label{sec:QuasiRetractions}
Let $X$ and $Y$ be metric spaces. For constants $C\geq 1$ and $D\geq 0$, a function $f\colon X \to Y$ is \textit{$(C,D)$-Lipschitz} if for any $x_1,x_2 \in X$, $\dist_Y(f(x_1),f(x_2))\leq C\dist_X (x_1,x_2)+D$. We say that $X$ is a \emph{quasi-retract} of $Y$ if there exist constants $C\geq 1$ and $D\geq0$ and $(C,D)$-Lipschitz functions $f\colon X \to Y$ and $g\colon Y \to X$ such that:
\[\dist_X(g(f(x)),x)\leq D, \]
for any $x\in X$. The pair $(f,g)$ is called a \emph{quasi-retraction} of $Y$ to $X$. If $X$ is a quasi-retract of $Y$, then we say that \emph{$Y$ quasi-retracts to $X$}.

A graph $\Delta$ is a \emph{quasi-retract} of a graph $\Gamma$ if $V(\Delta)$ is a quasi-retract of $V(\Gamma)$ as metric spaces with respect to the path distance.
\begin{example}
Let $\Gamma$, $\Delta$ be two connected graphs.
\begin{enumerate}
\item If $\Delta$ is finite, then $\Delta$ is a quasi-retract of $\Gamma$.
\item Recall that if $\Delta$ is a subgraph of $\Gamma$, then a graph retraction $r
\colon 
\Gamma \to \Delta$ is a function between the vertex sets that maps adjacent vertices to adjacent vertices and that it restricts to the identity on the vertex set of $\Delta$.  A graph retraction $r\colon \Gamma \to \Delta$ induces a quasi-retraction $(\imath, r)$ of $\Gamma$ to $\Delta$ where $\imath\colon \Delta \to \Gamma$ is the inclusion. However quasi-retractions into subgraphs are not necessarily graph retractions. For example, for the graph $\Gamma$ in Figure~\ref{fig:propo212}, if $\Delta$ is the subgraph induced by the vertices $a,b$, then there is no retract $\Gamma \to \Delta$; however, there is a quasi-retraction $(f,g)$ of $\Gamma$ into $\Delta$ where $g\colon V(\Gamma) \to V(\Delta)$ is given by $g(x)=b$ if $x\neq a$ and $g(a)=a$ and $f\colon \{a,b\} \to V(\Gamma)$ is the inclusion.
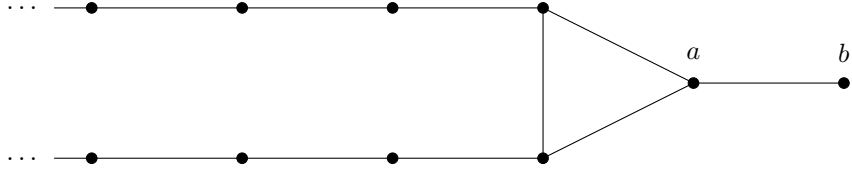
\begin{figure}[ht]
	\centering
\begin{tikzpicture}[scale=1]
\foreach \Point in {(2,1),(0,1),(-2,0),(-2,2),(-4,0),(-6,0),(-8,0),(-4,2),(-6,2),(-8,2)
}{
\filldraw \Point circle (2pt);
}
\draw (0,1)--(-2,0)--(-2,2)--cycle;
\draw (2,1)--(0,1);
\draw (-2,0)--(-8.5,0);
\draw (-2,2)--(-8.5,2);
\node at (-8.9,2) {$\dots$};
\node at (-8.9,0) {$\dots$};
\node at (2,1.4) {$b$};
\node at (0,1.4) {$a$};
		\end{tikzpicture}
	\caption{This graph does not retract to the subgraph induced by $\{a,b\}$, but it does quasi-retract.}
	\label{fig:propo212}
\end{figure}
\item \label{ProdQR} If $\Lambda$ denotes the Cartesian product $\sq{\Gamma}{\Delta}$, the strong product $\xsq{\Gamma}{\Delta}$, or the rooted product $\cir{\Gamma}{\Delta}{y}$; then $\Lambda$ quasi-retracts to $\Gamma$ and also to $\Delta$. The lexicographic product $\lx{\Gamma}{\Delta}$ quasi-retracts to $\Gamma$, but not necessarily to $\Delta$.

These statements have similar proofs. Let us illustrate the case of the Cartesian product. Recall that $\sq{\Gamma}{\Delta}$ is the graph with vertex set $V(\Gamma)\times V(\Delta)$ and such that vertices $(u,v)$ and $(u',v')$ are adjacent if either $u=u'$ and $v$ and $v'$ are adjacent in $\Delta$, or $v=v'$ and $u$ and $u'$ are adjacent in $\Gamma$.
Choose a vertex $v_0\in V(\Delta)$ and observe that the inclusion
$f\colon V(\Gamma) \hookrightarrow V(\sq{\Gamma}{\Delta})$ defined by $u\mapsto (u,v_0)$ preserves path distance, that is, $\dist_\Gamma(u,u')=\dist_{\Gamma\square\Delta}((u,v_0),(u',v_0))$. On the other hand, the projection $g\colon V(\sq{\Gamma}{\Delta}) \to V(\Gamma)$ defined by $(u,v)\mapsto u$ does not increase path distances, that is, if there is a  path of length $n$ in $\sq{\Gamma}{\Delta}$ between the vertices $(u,v)$ and $(u',v')$, then there is a path between $u$ and $u'$ in $\Gamma$ of length at most $n$. Hence both $f$ and $g$ are $(1,0)$-Lipschitz maps, and since $g\circ f$ is the identity map, it follows that $\sq{\Gamma}{\Delta}$ quasi-retracts to $\Gamma$.
\end{enumerate}
\end{example}
\begin{theorem}\label{corocopN1}
Let $\Gamma$ and $\Delta$ be connected graphs. If $\Delta$ is a quasi-retract of $\Gamma$, then
\[\wcop(\Delta) \leq \wcop(\Gamma) \quad \text{and} \quad \scop(\Delta) \leq \scop(\Gamma).\]
\end{theorem}

The objective of the rest of the section is to prove the above theorem.

\subsection{Convention on the notation for $\mathsf{CopWin}$.}\label{subsec:copwin-notation}
Let us introduce some notation that is used in the next sections. We also include a lemma for future reference.
Recall from the introduction that a graph $\Gamma$ is $\mathsf{CopWin}(n,\sigma,\rho,\psi,R)$ if for every vertex $v$,
$n$ cops with speed $\sigma$ and reach $\rho$ can eventually protect the $R$-ball centered at $v$. The point is that $\mathsf{CopWin}$ has {\bf five parameters} and its definition has a universal quantifier on the set of vertices of $\Gamma$.

Abusing notation, we use $\mathsf{CopWin}$ with {\bf six parameters} in the following sense.  If $v$ is a vertex of a graph $\Gamma$, we say \emph{$\Gamma$ is $\mathsf{CopWin}(n,\sigma,\rho,\psi, v, R)$} if $n$ cops with speed $\sigma$ and reach $\rho$ can eventually protect the $R$-ball centered at $v$ from a robber with speed $\psi$. In particular, 
\[  \text{$\Gamma$ is } \mathsf{CopWin}(n,\sigma,\rho,\psi, R)  \Longleftrightarrow \forall\ v:\ \text{$\Gamma$ is } \mathsf{CopWin}(n,\sigma,\rho,\psi,v, R).\]
With this notation,
the definitions of $n$-weak cop win and $n$-strong cop win can be re-stated as follows:
\begin{itemize}
\item \( \text{$\Gamma$ is $n$-weak cop win} \Longleftrightarrow \exists\ \sigma, \rho \ \forall\ \psi, R, v: \ \text{$\Gamma$ is } \mathsf{CopWin}(n,\sigma,\rho,\psi,v, R). \)
\item \( \text{$\Gamma$ is $n$-strong cop win} \Longleftrightarrow \exists\ \sigma, \forall\ \psi\ \exists\ \rho\ \forall\ R, v: \ \text{$\Gamma$ is } \mathsf{CopWin}(n,\sigma,\rho,\psi,v, R). \)
\end{itemize}
The following lemma eases the computation of cop numbers by allowing us to consider only balls around a single vertex instead of considering balls around arbitrary vertices.
\begin{lemma}
Let $\Gamma$ be a connected graph. The following statements are equivalent.
\begin{itemize}
\item \( \text{$\Gamma$ is $n$-weak cop win} \Longleftrightarrow \exists\ \sigma, \rho, v \ \forall\ \psi, R \ \colon \text{$\Gamma$ is } \mathsf{CopWin}(n,\sigma,\rho,\psi,v, R). \)
\item \( \text{$\Gamma$ is $n$-strong cop win} \Longleftrightarrow \exists\ \sigma, v\ \forall\ \psi\ \exists\ \rho\ \forall\ R \ \colon \text{$\Gamma$ is } \mathsf{CopWin}(n,\sigma,\rho,\psi,v, R). \)
\end{itemize}
\end{lemma}
\begin{proof}
Let us consider the first equivalence. The only if direction (left to right) is trivial. For the if direction, let $\sigma$, $\rho$ and $v$ satisfy  the statement. Let $\psi$ and $R$ be integers, and $u\in V(\Gamma)$. Let $r=\dist_\Gamma(u,v)$ and note that $\Gamma$ being connected implies $r<\infty$. By hypothesis, $\Gamma$ is $\mathsf{CopWin}(n,\sigma,\rho,\psi,v, R+r)$. Since the ball of radius $R+r$ around $v$ contains the ball of radius $R$ around $u$, it follows that $\Gamma$ is $\mathsf{CopWin}(n,\sigma,\rho,\psi,u, R)$. Since $u$ and $R$ were arbitrary, it follows that $\Gamma$ is $n$-weak cop win. The proof of the second equivalence is analogous.
\end{proof}
\subsection{Proof of Theorem~\ref{corocopN1}}
The theorem will follow from the proposition below.
\begin{proposition}\label{lemaqrCN}
Let $\Gamma$ and $\Delta$ be connected graphs, and let $(f, g)$ be a quasi-retraction with constants $(C,D)$ from $\Gamma$ to $\Delta$.
\noindent If $\Gamma$ is $\mathsf{CopWin}(n,\sigma,\rho,\psi,f(u_0), R)$, then $\Delta$ is $\mathsf{CopWin}(n,\sigma_\Delta,\rho_\Delta,\psi_\Delta, u_0,R_\Delta)$ cop win, where:
\[ \sigma_\Delta=C \sigma + D, \quad \rho_\Delta = C \rho + C^2 + CD +2D+C, \quad R=CR_{\Delta}+D,\quad \psi=(C+D)\psi_\Delta.\]
\end{proposition}

\begin{proof}
Fix the parameters $(n,\sigma,\rho,\psi, f(u_0),R)$ for the game in $\Gamma$, and let the  parameters for the game on $\Delta$ be defined as in the statement of the proposition. 

Suppose that $\Gamma$ is $\mathsf{CopWin}(n,\sigma,\rho,\psi,f(u_0), R)$. To show that the graph  $\Delta$ is $\mathsf{CopWin}(n,\sigma_\Delta,\rho_\Delta,\psi_\Delta,u_0,R_\Delta)$ we describe a winning strategy for $\Delta$ based on the winning strategy for $\Gamma$. We will be playing simultaneous games on $\Delta$ and $\Gamma$. The cops in $\Gamma$ will move according   to a winning strategy. The moves of the cops in $\Gamma$ will determine the moves of the cops in $\Delta$, and the move of the robber on $\Delta$ will determine the move of the robber on $\Gamma$. We will show that this translates into a winning strategy for the cops on $\Delta$.

In the set up of the game on $\Gamma$, the initial position of the $i$-th cop is denoted as $c_{ i,0 }^\Gamma$. For the game on  $\Delta$, choose the initial position of the $i$-th cop as 
\[c_{i,0}=g(c_{ i,0 }^\Gamma).\]
In this way, each cop in $\Delta$ corresponds to a cop in $\Gamma$, and this correspondence will be fixed throughout the game.

At this point the robber in $\Delta$ chooses an initial position $r_0 \in \Delta$. Let the initial position of the robber in $\Gamma$ be  \[r_0^\Gamma=f(r_0).\] 

For the game in $\Gamma$, let $c_{1,l}^\Gamma,c_{2,l}^\Gamma,\dots,c_{n,l}^\Gamma$ and $r_l^\Gamma$ denote the locations of the the cops and the robber, respectively, at the end of the $l$-stage of the game. Similarly for the game on $\Delta$, let $c_{1,l},c_{2,l},\dots,c_{n,l}$ and $r_l$ denote the locations of the the cops and the robber at the end of the   $l$-stage.  The movements of the cops in $\Delta$ are defined by 
   \[c_{i,l}=g(c_{i,l}^\Gamma), \]
for $1\leq i \leq n$. The movement from   $c_{i,l}$  to   $c_{i,l+1}$ during the $(l+1)$-stage is valid since 
   \[ \begin{split}
	  \dist_\Delta (c_{i,l},c_{i,l+1})= \dist_\Delta(g(c^\Gamma_{i,l}),g(c^\Gamma_{i,l+1}) ) & \leq C \dist_\Gamma(c^\Gamma_{i,l},c^\Gamma_{i,l+1})+D\\
	   & \leq C\sigma+D=\sigma_\Delta. 
	\end{split}
	\]
During the $(l+1)$-stage, after the cops have made their movement in $\Delta$, if the robber in $\Delta$ has not been captured, the robber  moves from $r_l$ to a position $r_{l+1}$. 
 
 Now the robber in $\Gamma$ moves from $r_l^\Gamma$ to 
 \[ r_{l+1}^\Gamma=f(r_{l+1}).\]
 We need to show that this is a valid move, that is, there is a path in $\Gamma$ from $r_{l}^\Gamma$ to $r_{l+1}^\Gamma$ of length at most $\psi$ which has every vertex at a distance larger than $\rho$ from every cop. 

Since the move of the robber from $r_l$ to $r_{l+1}$ in $\Delta$ was valid, there is a path $[r_{l}=w_0,w_1,\ldots , w_k=r_{l+1}]$ of length $k\leq \psi_\Delta$ such that $\rho_\Delta<\dist_\Delta(w_j, c_{i,l+1})$ for each $0\leq j\leq k$ and $1\leq i\leq n$. It follows that \[\begin{split}
 \rho_\Delta &<  \dist_\Delta(w_j, c_{i,l+1})\\ & \leq \dist(w_i, g(f(w_j))+ \dist_\Delta(g(f(w_j)), g(c_{i,l+1}^\Gamma))\\
 & \leq 2D+ C\dist_\Gamma(f(w_j), c_{i,l+1}^\Gamma) ) .
 \end{split}\]
Since $\rho_\Delta=C\rho+C^2+CD+2D+C$, we have that 
 \[ \rho+C+D+1 \leq \dist_\Gamma(f(w_j), c_{i,l+1}^\Gamma).\]
 On the other hand, 
 \[ \dist_\Gamma(f(w_j), f(w_{j+1}))\leq C\dist_\Delta(w_j, w_{j+1})+D =C+D.\]
 These last two inequalities imply that every vertex in a geodesic from $f(w_i)$ to $f(w_{i+1})$ is at a distance larger than $\rho$ from the cops positions $c_{i,l+1}^\Gamma$ during the $(l+1)$-stage. Hence a path from $r_l^\Gamma$ to $r_{l+1}^\Gamma$ obtained as a concatenation of geodesic paths  between consecutive vertices of the sequence $r^\Gamma_l=f(w_0),f(w_1),\ldots, f(w_k)=r_{l+1}^\Gamma$ has length at most 
 \[ (C+D)k \leq (C+D)\psi_\Delta = \psi \]
 and every vertex of this path is at a distance larger than $\rho$ from  the cop  positions $c_{i,l+1}^\Gamma$ during the $(l+1)$-stage. Hence the move of the robber in $\Gamma$ during the $(l+1)$-stage is valid.
 
Throughout the rest of the game, the moves of the cops in $\Delta$ are given by the moves of the cops in $\Gamma$, and the moves of the robber in $\Gamma$ are given by the  moves of the robber in $\Delta$ as described above.  The conclusion of the proposition then follows from the following two claims.

\begin{claim}
Once the robber is captured in the game on $\Gamma$, the robber will be captured in the game on $\Delta$.    
\end{claim}
Suppose that the robber is captured in $\Gamma$ during the  $(l+1)$-stage of the game. This means that $\dist_\Gamma (c_{ i,l }^\Gamma,r_{l} ^\Gamma)\leq \sigma+\rho$ for some $i$. It follows that
\[
\begin{split}
	   \dist_\Delta(c_{i,l},r_{l})
	   & = \dist_\Delta(g(c^\Gamma_{i,l}),r_{l} )\\
	   & \leq   \dist_\Delta(g(c^\Gamma_{i,l}),g(f(r_{l})) + \dist_\Delta(g(f(r_{l}), r_{l})\\  
	   & \leq C \dist_\Gamma(c^\Gamma_{i,l},f(r_{l}))+2D\\
  	   & = C \dist_\Gamma(c^\Gamma_{i,l}, r_{l}^\Gamma )+2D\\
	   & \leq C(\sigma+\rho) +2D\\
	   & = C\sigma + C\rho+2D \\&< \sigma_\Delta+\rho_\Delta 
\end{split}
\]	
This shows that at the end of the $l$-stage the robber in $\Delta$ is at distance less than $\sigma_\Delta+\rho_\Delta$ from at least one cop, and hence the robber in $\Delta$ is captured during the $(l+1)$-stage.

\begin{claim}
If the robber is forced out of the $R$-ball around $f(u_0)$ in the game on $\Gamma$, then the robber will be forced out of the $R_\Delta$-ball around $u_0$ in the game on $\Delta$.    
\end{claim}
Assume the robber is   outside the $R$-ball around $f(u_0)$ in the game on $\Gamma$ at the end of the $l$-stage, that is, 
	$\dist_\Gamma(r^\Gamma_{l}, f(u_0)) > R$. Then 
\[
\begin{split}
	   \dist_\Delta(r_l, u_0) & \geq \frac{1}{C} \left( \dist_\Gamma(f(r_l), f(u_0)) - D \right)\\
	   & = \frac{1}{C} \left( \dist_\Gamma( r_l^\Gamma , f(u_0)) - D\right)\\
       & \geq \frac{1}{C} ( R  - D)\\
       & = R_\Delta
\end{split}
\]	
Hence, if the robber in $\Gamma$ is outside  the ball $B(f(u_0), R)$ in $\Gamma$, then  the robber in $\Delta$ is outside the ball $B(u_0, R_\Delta)$. 
\end{proof}

\begin{proof}[Proof of Theorem~\ref{corocopN1}.]
Let $\Gamma$ and $\Delta$ be connected graphs, and let  $(f, g)$ be a quasi-retraction  from $\Gamma$ to $\Delta$ with constants $(C,D)$. Suppose that $\Gamma$ is  $n$-weak cop win. Then there are $\sigma$ and $\rho$ such that such that $\Gamma$ is 
$\mathsf{CopWin}(n,\sigma,\rho,  \psi, v_0, R)$ for every vertex $v_0$, and any $\psi>0$ and $R>0$.  Let $\sigma_\Delta=C\sigma+D$ and 
$\rho_\Delta = C \rho + C^2 + CD +2D+C$.  By Proposition~\ref{lemaqrCN}, $\Delta$ is $\mathsf{CopWin}(n,\sigma_\Delta,\rho_\Delta, \psi_\Delta, u_0, R_\Delta)$ for any vertex $u_0$, and any $\psi_\Delta>0$ and $R_\Delta>0$. In particular, $\wcop(\Delta)\leq n$. 

The argument proving that $\scop(\Delta) \leq \scop(\Gamma)$ is analogous.
\end{proof}

\section{Weak cop number of one-ended non-amenable graphs}\label{sectionamenable}
Let $\Gamma$ be a connected, locally finite graph. 
The graph $\Gamma$ is \emph{one-ended} if for any finite subset of vertices $K$ the induced subgraph $\Gamma-K$ has only one unbounded connected component.
For a subset of vertices $K$ of $\Gamma$, let $\partial K$ be the set of edges of $\Gamma$ with one endpoint in $K$ and the other endpoint not in $K$. The \emph{Cheeger constant} of $\Gamma$ is defined as
\[h(\Gamma)=\inf \left\{ \frac{|\partial K|}{|K|}\mid \text{$K$ is a non-empty finite subset of vertices of $\Gamma$}\right\}.\]
The graph $\Gamma$ is \emph{non-amenable} if it has nonzero Cheeger constant. Observe that if $\Gamma$ is a non-amenable connected locally finite graph then $\Gamma$ is an infinite graph. 

An \emph{automorphism} of a graph $\Gamma$ is a bijection $V(\Gamma)\to V(\Gamma)$ such that any two vertices are adjacent in $\Gamma$ if and only if their  images are adjacent. The set of automorphisms, denoted by $\mathsf{Aut}(\Gamma)$, is a group under composition of functions. A graph $\Gamma$ is \emph{vertex-transitive} if   $\mathsf{Aut}(\Gamma)$  acts transitively on its vertex set.

\begin{theorem}\label{1eN-A}
If $\Gamma$ is a connected, locally finite, one-ended, non-amenable, vertex transitive graph, then $\wcop(\Gamma)=\infty$.
\end{theorem}

For the rest of the section, assume that $\Gamma$ is a connected, locally finite, vertex transitive graph.
Let us introduce some notation. If $K$ is a subset of vertices of $\Gamma$, then $\Gamma-K$ denotes the subgraph of $\Gamma$ induced by $V(\Gamma)-K$.
If $\Delta$ is a subgraph of $\Gamma$, then $|\Delta|$ denotes the number of vertices of $\Delta$, $\partial \Delta$ denotes $\partial V(\Delta)$, and $\Gamma-\Delta$ is the subgraph $\Gamma-V(\Delta)$.

\subsection{Subgraph undistorted embedding}
In the case that $\Delta$ is a finite subgraph of a connected graph $\Gamma$ such that the induced subgraph $\Gamma-\Delta$ is connected, then it is an exercise to show that there is a constant $L$ such that $\dist_{\Gamma-\Delta}(a,b) \leq L\dist_\Gamma(a,b)$ for any pair of vertices $a,b\in \Gamma-\Delta$. Roughly speaking, $L$ is the diameter of $\partial \Delta$ in $\Gamma-\Delta$, and to prove the inequality one constructs a path in $\Gamma-\Delta$ from $a$ to $b$ by taking a geodesic in $\Gamma$ from $a$ to $b$ and replacing its subpaths intersecting $\Delta$ by subpaths in $\Gamma-\Delta$ of length $L$ that go around $\Delta$. The following proposition uses a similar argument to obtain a better estimation of $L$ by taking into account connectivity information of $\Delta$.
\begin{proposition}[Undistorted Embedding ]\label{prop:RobberSpeed}
Let $\Gamma$ be a connected, vertex transitive, locally finite graph. For any pair of integers $m$ and $n$, there is an integer $L_{m,n}=L(\Gamma,m,n)$ with the following property.
Let $\Delta$ be a subgraph of $\Gamma$ with at most $m$ vertices and such that $\Gamma-\Delta$ is connected. Suppose that $\Delta$ is the union of pairwise disjoint connected subgraphs $\Delta_1, \Delta_2, \ldots , \Delta_n$. If $a,b$ are vertices in $\Gamma-\Delta$, then
\[ \dist_{\Gamma-\Delta}(a,b) \leq L_{m,n}\dist_\Gamma (a,b).\]
\end{proposition}
The proof of the proposition is divided into lemmas. The definition of the constants $L_{m,n}$ are part of Lemma~\ref{lem:DefLmn}. That the constants $L_{m,n}$ satisfy the statement of the proposition is proved by induction on $n$, where Lemma~\ref{lem:BaseCase} provides the case $n=1$, and then Lemma~\ref{lem:InductiveStep} concludes the proof of the proposition.
Recall that the distance between subgraphs $\Delta_1$ and $\Delta_2$ of $\Gamma$ is defined as \[ \dist_\Gamma(\Delta_1, \Delta_2):=\min\{\dist_\Gamma(u_1,u_2)\mid u_1\in V(\Delta_1) \text{ and } u_2\in V(\Delta_2)\}.\]
\begin{lemma}\label{lem:counting}
Let $m,n$ and $L$ be positive integers. Then, up to the action by $\mathsf{Aut}(\Gamma)$, there are finitely many subgraphs $\Delta$ such that
\begin{itemize}
\item $|\Delta|\leq m$,
\item $\Delta$ is the union of $n$ pairwise disjoint connected subgraphs $\Delta=\Delta_1\cup \dots \cup \Delta_n$,
\item $0< \dist_\Gamma (\Delta_1, \Delta_i) \leq L$ for $1<i\leq n$.
\end{itemize}
\end{lemma}
\begin{proof}
Fix a vertex $u$ of $\Gamma$. Since $\Gamma$ is locally finite, there are finitely many subgraphs $\Delta$ as in the statement that contain the vertex $u$. Since $\Gamma$ is vertex transitive, the statement of the lemma follows.
\end{proof}

\begin{definition}Let $\Gamma$ be a graph and let $\Delta_1, \dots , \Delta_n$ be a collection of $n$ subgraphs. Define $\mathsf{Rips}_\Gamma(\Delta_1, \dots , \Delta_n; L)$ as the graph with vertex set $\{\Delta_1, \dots , \Delta_n\}$ and edge set $\{ \{\Delta_i, \Delta_j\} \mid 0< \dist_\Gamma (\Delta_i, \Delta_j) \leq L \}$. See Figure~\ref{AGraph} for an illustration of this definition.

\begin{figure}[ht]
	\centering
		\begin{tikzpicture}[scale=0.6]
\foreach \Point in { (0,2), (3,4),(-3.5,2),(1.5,2.5), (-4.1,3.5),(1,4), (-1.6,4.5)}{
\filldraw[color=black!60, fill=black!3, very thick,opacity=0.6]\Point circle (1.2);
}
\node at (3,4) {$\Delta_1$};
\node at (1.5,2.5) {$\Delta_2$};
\node at (1,4) {$\Delta_3$};
\node at (0,2) {$\Delta_4$};
\node at (-3.5,2) {$\Delta_5$};
\node at (-4.1,3.5) {$\Delta_6$};
\node at (-1.6,4.5) {$\Delta_7$};
\foreach \Point in { (12,2), (15,4),(9,1.5),(14,2), (7.9,3.5),(13,4), (10.4,4.5)}{
\filldraw \Point circle (2pt);
}
\node at (15,4.6) {$v_1$};
\node at (14,1.4) {$v_2$};
\node at (13,4.6) {$v_3$};
\node at (12,1.4) {$v_4$};
\node at (9,0.9) {$v_5$};
\node at (7.9,4.1) {$v_6$};
\node at (10.4,5.1) {$v_7$};
\draw (15,4)--(14,2) --(13,4) -- (12,2)--(14,2);
\draw (15,4)-- (13,4);
\draw (7.9,3.5) -- (9,1.5);
		\end{tikzpicture}
	\caption{Schematic of a graph $\mathsf{Rips}_\Gamma(\Delta_1,\ldots ,\Delta_7; L)$}
	\label{AGraph}
\end{figure}
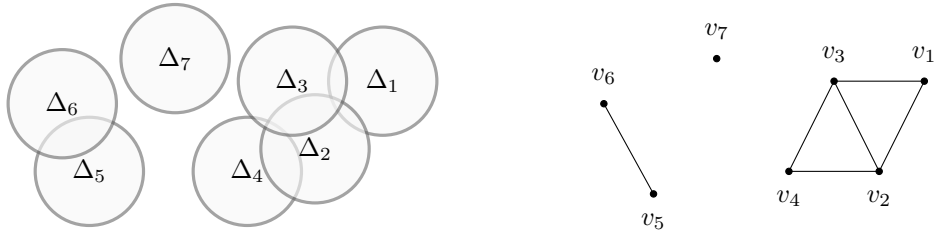
\end{definition}
\begin{remark}\label{rem:separation}
If $\mathsf{Rips}_\Gamma(\Delta_1, \dots , \Delta_n; L)$ is disconnected then, after re-enumerating the $\Delta_i$'s, there is $1\leq k< n$ such that
$\mathsf{Rips}_\Gamma(\Delta_1, \dots , \Delta_n; L)$ is the disjoint union of
$\mathsf{Rips}_\Gamma(\Delta_1, \dots , \Delta_k; L)$ and
$\mathsf{Rips}_\Gamma(\Delta_{k+1}, \dots , \Delta_n; L)$.
\end{remark}
\begin{lemma}\label{lem:counting2}
Let $m,n$ and $L$ be positive integers. Then, up to the action by $\mathsf{Aut}(\Gamma)$, there are finitely many subgraphs $\Delta$ such that
\begin{itemize}
\item $|\Delta|\leq m$,
\item $\Delta$ is the union of $n$ pairwise disjoint connected subgraphs $\Delta=\Delta_1\cup \dots \cup \Delta_n$,
\item $\mathsf{Rips}_\Gamma(\Delta_1, \dots , \Delta_n; L)$ is connected.
\end{itemize}
\end{lemma}
\begin{proof}
Since each $\Delta_i$ is connected with at most $m$ vertices, it follows that each $\Delta_i$ has diameter at most $m-1$. Since $\mathsf{Rips}_\Gamma(\Delta_1, \dots , \Delta_n; L)$ is connected, we have that $0<\dist_\Gamma(\Delta_1, \Delta_i) \leq (m+n)L$ for $1<i\leq n$. Then the statement follows from Lemma~\ref{lem:counting}.
\end{proof}
\begin{lemma}\label{lem:DefLmn}
For integers $m>0,n\geq 0$, let $L_{m,0}=0$
and
\[ \begin{split}
L_{m,n} = \max \{ & \dist_{\Gamma-\Delta} (a,b) \mid \\
&1\leq k\leq n,\\
& \text{$\Delta_1, \ldots, \Delta_k$ are disjoint connected subgraphs of $\Gamma$, } \\
& \text{$\Delta=\Delta_1\cup\ldots \cup\Delta_k$,}\\
& \mathsf{Rips}_\Gamma(\Delta_1, \dots , \Delta_k; L_{m,k-1}) \text{ is connected},\\
& \text{$\Gamma-\Delta$ is connected,}\\
& \text{$|\Delta|\leq m$,} \\ & \text{ $a,b\in V(\Gamma)-V(\Delta)$, and}\\
& \text{$\dist_\Gamma(a, \Delta)=1$ and $\dist_\Gamma(b,\Delta)=1$} \}.
\end{split} \]
Then $L_{m,n}$ are well-defined integers such that
\[0\leq L_{m,1} \leq L_{m,2 }\leq\cdots L_{m, k}\leq L_{m,k+1}\leq \cdots. \]
\end{lemma}
\begin{proof}
By Lemma~\ref{lem:counting2}, for $m>0$ and $n>0$, $L_{m,n}$ is the maximum of a finite number of integers and hence it is well-defined. That $L_{m,k}\leq L_{m,k+1}$ is immediate from the definition.
\end{proof}
\begin{lemma}\label{lem:BaseCase}
For any connected subgraph $\Delta$ such that $|\Delta|\leq m$ and $\Gamma-\Delta$ is connected, if $a,b$ are vertices in $\Gamma-\Delta$ then
\[ \dist_{\Gamma-\Delta}(a,b) \leq L_{m,1}\dist_\Gamma (a,b).\]
\end{lemma}
\begin{proof}
Let $v$ and $w$ be vertices of $\Gamma-\Delta$ and let $p$ be a geodesic in $\Gamma$ between them. Then one can replace maximal subpaths of $p$ with all interior vertices in $\Delta$ with subpaths in $\Gamma-\Delta$ of length at most $L_{m,1}$ showing that $\dist_{\Gamma-\Delta}(v,w)\leq L_{m,1} \dist_\Gamma(v,w)$.
\end{proof}
\begin{lemma}\label{lem:InductiveStep}
Let $\Delta$ be a subgraph
such that $|\Delta|\leq m$, $\Gamma-\Delta$ is connected, and $\Delta$ is the union of $n$ pairwise disjoint connected subgraphs $\Delta=\Delta_1\cup \dots \cup \Delta_n$. If $a,b$ are vertices in $\Gamma-\Delta$ then
\[ \dist_{\Gamma-\Delta}(a,b) \leq L_{m,n}\dist_\Gamma (a,b).\]
\end{lemma}
\begin{proof}
Fix $m$. We argue by induction on $n$ that $L_{m,n}$ satisfies the property of the statement of the lemma. The base case $n=1$ is Lemma~\ref{lem:BaseCase}. Suppose $n>1$ and that the result holds for all $L_{m,k}$ if $n>k$. We consider two cases.
\emph{Case 1.} Suppose that $\mathsf{Rips}_\Gamma(\Delta_1,\ldots ,\Delta_n; L_{m, n-1})$ is disconnected. By Remark~\ref{rem:separation}, we can assume that there is $1\leq k< n$ such that the graph $\mathsf{Rips}_\Gamma(\Delta_1, \dots , \Delta_n; L_{m, n-1})$ is the disjoint union of $\mathsf{Rips}_\Gamma(\Delta_1, \dots , \Delta_k; L_{m, n-1})$ and
$\mathsf{Rips}_\Gamma(\Delta_{k+1}, \dots , \Delta_n; L_{m, n-1})$.
Let $\Lambda_1 = \Delta_1 \cup \cdots \cup \Delta_k$ and $\Lambda_2= \Delta_{k+1}\cup \cdots \cup \Delta_n$, and observe that
\[ \dist_\Gamma(\Lambda_1, \Lambda_2) >L_{m, n-1}.\]
Note that the subgraph $\Gamma-\Lambda_i$ is connected, $|\Lambda_i|\leq m$, and $\Lambda_i$ has at most $n-1$ connected components. By induction, the constant $L_{m,n-1}$ applies to both $\Lambda_i$. More specifically, suppose that $a,b \in V(\Gamma)$ satisfy that $\dist(a,\Lambda_1)=\dist(b,\Lambda_1)=1$. By definition of $L_{m,n-1}$, any geodesic path $q$ in $\Gamma-\Lambda_1$ between $a$ and $b$ has length at most $L_{m,n-1}$ and hence it does not intersect the subgraph $\Lambda_2$. It follows that $q$ is a path in $\Gamma-\Delta$. This last reasoning is symmetric in $\Lambda_1$ and $\Lambda_2$.
Let $v$ and $w$ be vertices of $\Gamma-\Delta$ and let $p$ be a geodesic in $\Gamma$ between them. Then each maximal subpath of $p$ with all internal vertices in $\Lambda_1$ (respectively, $\Lambda_2$) can be replaced with a subpath in $\Gamma-\Delta$ of length at most $L_{m, n-1}$. In this way, one can obtain a path between $v$ and $w$ of length at most $L_{m, n-1}\dist_\Gamma(v,w)$ in $\Gamma-\Delta$. Hence
\[\dist_{\Gamma-\Delta}(v,w)\leq L_{m,n-1} \dist_\Gamma(v,w) \leq L_{m,n} \dist_\Gamma (v,w).\]
\emph{Case 2.} Suppose that $\mathsf{Rips}_\Gamma(\Delta_1,\ldots ,\Delta_n; L_{n-1})$ is connected.
Let $v$ and $w$ be vertices of $\Gamma-\Delta$. Let $p$ be a geodesic path in $\Gamma$ between $v$ and $w$.
Then, by definition of $L_{m,n}$, one can replace each maximal subpath of $p$ with all interior vertices in $\Delta$ with subpaths in $\Gamma-\Delta$ of length at most $L_{m,n}$. This produces a path in $\Gamma-\Delta$ between $v$ and $w$ of length at most $L_{m,n} \dist_\Gamma(v,w)$, and hence $\dist_{\Gamma-\Delta}(v,w)\leq L_{m,n} \dist_\Gamma(v,w)$.
\end{proof}
\subsection{Proof of Theorem~\ref{1eN-A}}
For the rest of the section, let $\Gamma$ be a {\bf connected, vertex transitive, locally finite, one-ended and non-amenable} graph with Cheeger constant $h(\Gamma)$.
The proof of the theorem is divided into a sequence of lemmas. The hypothesis that $\Gamma$ is one-ended is only used in Lemma~\ref{lem:ValidMove}, while the other hypotheses on $\Gamma$ (in particular being non-amenable) are used throughout most of the argument. The section concludes by putting all the lemmas together to deduce the statement of the theorem.
\begin{remark}[Estimations using the Cheeger constant]\label{remark1}
Let $K$ be a finite subset of vertices and consider the subgraph $\Gamma-K$. If $\Delta$ is a connected component of $\Gamma-K$ then $\partial \Delta \subset \partial K$. Therefore
\begin{enumerate}
\item If $\Delta$ is a finite connected component, then
\[ |\Delta|\leq \frac{1}{h(\Gamma)} |\partial \Delta| \leq \frac1{h(\Gamma)} |\partial K| . \]
\item The number of connected components of $\Gamma-K$ is at most $|\partial K|$. As a consequence, the number of vertices of $\Gamma-K$ that belong to a finite connected component is at most $\frac1{h(\Gamma)}|\partial K|^2$.
\item If $\Delta$ is a connected subgraph of $\Gamma$ disjoint from $K$,
and $|\Delta|>\frac1{h(\Gamma)}|\partial K|$, then $\Delta$ is a subgraph of an unbounded connected component of $\Gamma-K$.
\end{enumerate}
\end{remark}
\begin{remark}[Notation for cardinality of spheres and balls]
Let $v$ be a fixed vertex of $\Gamma$. Let $\alpha (n)$ denote the number of vertices of $\Gamma$ at distance exactly $n$ from $v$, and let $\beta(n)$ denote the number of vertices of $\Gamma$ at distance at most $n$ from $v$. Since $\Gamma$ is vertex transitive, $\alpha(n)$ and $\beta(n)$ are independent of the choice of $v$, and we refer to them as the \emph{size of the spheres of radius $n$}, and the \textit{size of the balls of radius $n$}, respectively.
Since $\Gamma$ is necessarily an infinite graph, we have that $\beta\colon \ZZ_+ \to \ZZ_+$ is an increasing function and in particular $\lim_{n\to \infty}\beta(n)=\infty$.
\end{remark}
\begin{lemma}\label{lem:ValidMove}
For any positive integers $n$ and $\rho$, there is an integer $m=m(\Gamma, n,\rho)>0$ with the following property. If $\{c_1,\ldots, c_n\}$
is any collection of $n$ vertices of $\Gamma$, and $\Lambda$ is the
unbounded connected component of the subgraph $\Gamma-\bigcup_{i=1}^nB_\rho(c_i)$, then \[\dist_{\Lambda}(a,b) \leq L_{m,n}\dist_\Gamma (a,b)\]
for any pair of vertices $a,b$ of $\Lambda$, where
\[ m = n\beta(\rho) + \frac{( \alpha(\rho) \cdot d \cdot n )^2}{h(\Gamma)}\]
and $d$ is the degree of each vertex in $\Gamma$.
\end{lemma}
\begin{proof}
Let $\Delta$ be the subgraph of $\Gamma$ induced by all vertices in $\bigcup_{i=1}^nB_\rho(c_{i})$ together with all vertices that belong to bounded components of $\Gamma-\bigcup_{i=1}^nB_\rho(c_{i})$. Since $\Gamma$ is one-ended, the graph $\Gamma-\Delta$ is the unbounded connected component of $\Gamma-\bigcup_{i=1}^nB_\rho(c_{i})$. In particular, $\Lambda=\Gamma-\Delta$ is connected.
Let us argue that $\Delta$ has at most $m$ vertices. Observe that
\[ \left|\bigcup_{i=1}^n B_\rho(c_i)\right|\leq n\beta(\rho) \quad \text{and} \quad \left|\partial \bigcup_{i=1}^n B_\rho(c_i)\right|\leq \alpha(\rho) \cdot d \cdot n. \]
Hence, by Remark~\ref{remark1}, the number of vertices in $\Gamma - \bigcup_{i=1}^n B_\rho(c_i)$ that belong to bounded components is at most $\frac{1}{h(\Gamma)}\left( \alpha(\rho) \cdot n\cdot d\right)^2$ and therefore
\[ |\Delta| \leq n\beta(\rho) + \frac{1}{h(\Gamma)}\left( \alpha(\rho) \cdot d \cdot n)\right)^2 = m. \]
Now we argue that $\Delta$ has at most $n$ connected components. First observe that the subgraph of $\Gamma$ induced by a ball $B_\rho(c_i)$ is connected. Since $\Gamma$ is connected, every connected component of $\Gamma - \bigcup_{i=1}^n B_\rho(c_i)$ has a vertex adjacent to a vertex in $\bigcup_{i=1}^n B_\rho(c_i)$. It follows that every vertex of $\Delta$ is in a connected component containing a ball $B_\rho(c_i)$, and therefore $\Delta$ has at most $n$ connected components.
To summarize $\Delta$ is a subgraph of $\Gamma$ with at most $m$ vertices, at most $n$ connected components, and such that $\Gamma-\Delta$ is a connected subgraph. By Proposition~\ref{prop:RobberSpeed}, for any vertices $a,b\in \Gamma-\Delta$, we have that
\[ \dist_\Lambda(a,b)= \dist_{\Gamma-\Delta}(a,b)\leq L_{m,n}\dist_\Gamma(a,b).\qedhere\]
\end{proof}
\begin{lemma}[Safe Distance]\label{lem:lambda}
For any integers $n$ and $\rho$, there exists an integer $\lambda(\rho,n)$ such that for any collection of $(n+1)$ vertices, denoted as $r$ and $(c_1,\ldots,c_n)$, if $\dist(r,c_i)>\lambda(\rho,n)$ for every $i$, then $r$ lies in an unbounded component of $\Gamma-\bigcup_{i=1}^n B_{\rho}(c_{i})$.
\end{lemma}
\begin{proof}
Let $d$ be the degree of vertices of $\Gamma$.
Define $\lambda(\rho, n)=N$ as the least integer such that \[\beta (N)> \frac1{h(\Gamma)} (\alpha (\rho) \cdot d \cdot n).\]
Let $r$ and $(c_1,\ldots , c_n)$ be $n+1$ arbitrary vertices of $\Gamma$ such that $\dist(r,c_i)>\lambda(\rho,n)$ for every $i$.
Let $K=\bigcup_{i=1}^n B_\rho (c_i)$ and observe that
\[
\left|\partial K\right| \leq \sum_{i=1}^n |\partial B_{\rho}(c_i)|\leq n\cdot d \cdot \alpha(\rho) .
\]
Let $\Delta$ be the subgraph induced by $B_N (r)$. Observe that $\Delta$ is connected and $|\Delta|=\beta(N)> \frac1{h(\Gamma)}|\partial K|$. By the third item of Remark~\ref{remark1}, the subgraph $\Delta$ lies inside the unbounded connected component of $\Gamma-K$ that contains the vertex $r$.
\end{proof}
\begin{lemma}[Robber's strategy safe points]\label{lem:spoint}
Let $u_0$ be a fixed vertex. For any positive integers $n$, $\sigma$ and $\rho$, there exist integers $R>0$ and $D>\rho$, and there are $n+1$ vertices $\{s_1,\ldots, s_{n+1} \}$ such that the following properties hold:
\begin{enumerate}
\item For every $i$, $\dist(u_0,s_i)\leq R$
\item For any collection of $n$ vertices $\{c_1,\ldots, c_n \}$,
there is a vertex $s$ in $\{s_1,\ldots, s_{n+1} \}$ such that $\dist (s ,\{c_1,\ldots, c_n\})>D+\sigma$.
\item If $s\in \{s_1,\ldots, s_{n+1} \}$ satisfies $\dist (s ,\{c_1,\ldots, c_n\})>D$, then $s$ belongs to an unbounded component of $\Gamma-\bigcup_{i=1}^n B_{\rho}(c_i)$.
\end{enumerate}
\end{lemma}
\begin{proof} Let $\lambda(\rho, n)$ be the constant provided by Lemma~\ref{lem:lambda}. Let \[D>\max \{ \lambda(\rho,n), \rho \}. \]
Since $\Gamma$ is a locally finite, infinite and connected graph, we can let $s_1=u_0$ and choose inductively vertices $s_i$ such that $\dist(s_i,\{s_1,\ldots ,s_{i-1}\})> 2D+2\sigma$ to obtain a collection of $n+1$ vertices $s_1,\ldots ,s_n, s_{n+1}$ with the property that $\dist(s_i, s_j)> 2D+2\sigma$ if $i\neq j$. Let $R=\max\{\dist(u_0, s_i)\mid 1\leq i\leq n+1\}$. The first item of the lemma is immediate.
Let $\{c_1,\ldots ,c_n\}$ be $n$ vertices. First we prove that there is $s\in\{s_1,\ldots ,s_n\}$ such that $\dist(s,\{c_1,\ldots, c_n\})>D+\sigma$. Suppose by contradiction that this is not the case, that is, for each $1\leq i\leq n+1$, there exists at least one $c_j$ such that $\dist (s_i,c_j)\leq D+\sigma$. By the pigeon-hole argument, there must be $c_i$ and two distinct $s_j$ and $s_k$ such that $\dist (s_j,c_i)\leq D+\sigma$ and $\dist (s_k,c_i)\leq D+\sigma$. It follows that
\[ \dist(s_j,s_k)\leq \dist(s_j,c_i) + \dist (s_k,c_i) \leq 2D+2\sigma, \]
which contradicts the properties of the set $\{s_1, \ldots , s_{n+1}\}$.
Let $s\in\{s_1,\ldots ,s_n\}$ be such that $\dist(s,c_i)>D$ for $1\leq i\leq n$. Since $D>\lambda_{\Gamma}(\rho,n)$, Lemma~\ref{lem:lambda} implies that $s$ lies in an unbounded component of $\Gamma-\bigcup_{i=1}^n B_{\rho}(c_{i})$.
\end{proof}
\begin{proof}[Proof of Theorem~\ref{1eN-A}]
Let $\Gamma$ be the given graph and let $\sigma,\rho,u_0$, be fixed parameters. We will prove that for every $n$ there exists an $R$ and $\psi$ such that $\Gamma$ is not $n$-weak cop win. Let $n>0$. First, let us define the parameters $R$ and $\psi$.
Consider a collection of vertices $\{s_1, \ldots , s_{n+1}\}$ and the integers $R$ and $D$ provided by Lemma~\ref{lem:spoint} for the parameters $u_0,n,\sigma, \rho$. Let
\[ \psi = 2R L_{m,n}\]
where $L_{m,n}$ is the constant provided by Proposition~\ref{prop:RobberSpeed} for
\[ m = n\beta(\rho) + \frac{( \alpha(\rho)\cdot n\cdot d)^2}{h(\Gamma)}.\]
Consider the game in $\Gamma$ with parameters $(n,\sigma,\rho,\psi,u_0,R)$.
We will prove that the robber has a strategy such that it is never captured, and at any stage of the game his position is a vertex in $\{s_1, \ldots , s_{n+1}\}$. Since $\dist_\Gamma(u_0, \{s_1, \cdots , s_{n+1}\})<R$, this will be a winning strategy for the robber.
The robber will move in the following way:
\begin{itemize}
\item Suppose the $n$ cops have chosen their initial positions, say $c_{1,0}, \ldots c_{n,0}$.
By the second item of Lemma~\ref{lem:spoint}, there is a vertex
\[r_0\in \{s_1,\ldots ,s_{n+1}\}\] such that
\[\dist_\Gamma(s_i,\{c_{1,0},\ldots,c_{n,0} \} )>D+\sigma.\]
Let $r_0$ be the initial position of the robber. Since $\Gamma$ is one-ended, $\Gamma-\bigcup_{i=1}^n B_\rho(c_{i,0})$ has one unbounded connected component. The third item of Lemma~\ref{lem:spoint} implies that $r_0$ is in the unbounded component of $\Gamma-\bigcup_{i=1}^n B_\rho(c_{i,0})$.
\item Let $r_k$ and $c_{1,k},\ldots ,c_{n,k}$ denote the positions of the robber and the cops at the end of the $k$-stage, respectively. Suppose that
\[r_k \in \{s_1,\ldots ,s_{n+1}\}, \]
and \[\dist_\Gamma(r_k,\{c_{1,k},\ldots,c_{n,k} \} )>D+\sigma.\]
Then, at the beginning of the $(k+1)$-stage, the cops move first. The last inequality implies that \[\dist_\Gamma(r_k,\{c_{1,k+1},\ldots,c_{n,k+1} \} )>D>\rho\]
and hence the robber has not been captured.
By the second item of Lemma~\ref{lem:spoint}, there is a vertex $r_{k+1}\in \{s_1,\ldots ,s_{n+1}\}$ such that \[\dist_\Gamma(r_{k+1},\{c_{1,k+1},\ldots,c_{n,k+1} \} )>D+\sigma>\rho.\]
Now we argue that the robber has a valid move from $r_k$ to $r_{k+1}$.
Since $\Gamma$ is one-ended, $\Gamma-\bigcup_{i=1}^n B_\rho(c_{i,k+1})$ has only one unbounded connected component that we denote by $\Lambda$.
The last two inequalities of the previous paragraph together with the third item of Lemma~\ref{lem:spoint} imply that $r_k$ and $r_{k+1}$ are elements of the unbounded component $\Lambda$ of the subgraph $\Gamma-\bigcup_{i=1}^nB_\rho(c_{i,k+1})$.
By Lemma~\ref{lem:ValidMove}, we have that
\[ \dist_\Lambda(r_k, r_{k+1}) \leq L_{m,n} \dist_\Gamma(r_k, r_{k+1})\leq 2R\cdot L_{m,n}=\psi\]
where the second inequality follows from $\dist_\Gamma(u_0, \{s_1,\ldots ,s_{n+1}\})<R$. Hence there is a path in $\Lambda \subset \Gamma-\bigcup_{i=1}^n B_\rho(c_{k+1,i})$ from $r_k$ to $r_{k+1}$ of length at most $\psi$, and hence this is a valid move for the robber. \qedhere
\end{itemize}
\end{proof}
\section{\boldmath\texorpdfstring{$\Theta_n$}{Theta n}-extensions}
\label{sec:ThetaExtensions}
Let $\Gamma$ be a connected graph, let $u_0$ be a vertex of $\Gamma$, and let $n$ be a positive integer. The \emph{$\Theta_n$-extension of $\Gamma$ (centered at $u_0$)}, denoted as $\Theta_n(\Gamma)$, is the graph constructed as follows. Let $\Gamma_k$ for $1 \leq k\leq n$ be disjoint graphs isomorphic to $\Gamma$, and let $\eta_k\colon \Gamma \to \Gamma_k$ be a fixed isomorphism. Then $\Theta_n(\Gamma)$ is the disjoint union of $\Gamma_1\sqcup\cdots \sqcup \Gamma_{n}$ together with a collection of paths: a path (called a bridge) between $\eta_i(x)$ to $\eta_j(x)$ of length $\dist_\Gamma(u_0,x)+1$ for each vertex $x$ of $\Gamma$ and each $1\leq i<j\leq n$. See Figures~\ref{fig:Theta2Z} and~\ref{fig:Theta3Z} for some examples.
More formally, the graph $\Theta_n(\Gamma)$ is defined as follows. For $1\leq i<j\leq n$ and $x\in V(\Gamma)$, let $l=\dist_\Gamma(u_0, x)$,
let
\[ V_{i,j,x}=\{ x^{i,j}_1, \ldots x^{i,j}_l \} \]
and let
\[ E_{i,j,x}=\left\{ \{\eta_i(x),x^{i,j}_1\},
\{x^{i,j}_1, x^{i,j}_2 \}, \ldots \{x^{i,j}_{l-1}, x^{i,j}_l \} ,\{x^{i,j}_l, \eta_j(x)\} \right\}.\]
If $l=0$ then $V_{i,j,x}=\emptyset$ and $E_{i,j,x}$ consists of the single edge $\{\eta_i(x),\eta_j(x)\}$.
Then the vertex set and edge set of $\Theta_n(\Gamma)$ are defined as the disjoint unions
\[ V(\Theta_n(\Gamma)) = \bigsqcup_{i=1}^{n} V(\Gamma_i) \sqcup \bigsqcup_{x\in V(\Gamma)}\bigsqcup_{1\leq i<j\leq N} V_{i,j,x} \]
and
\[ E(\Theta_n(\Gamma)) = \bigsqcup_{i=1}^{n} E(\Gamma_i) \sqcup \bigsqcup_{x\in v(\Gamma)}\bigsqcup_{1\leq i<j\leq n} E_{i,j,x}. \]
For each $x\in \Gamma$, the path $[\eta_i(x), x^{i,j}_1, \ldots , x^{i,j_k}, \eta_j(x)]$ is called the \emph{bridge} between $\eta_i(x)$ and $\eta_j(x)$. In particular, any vertex of $\Theta_n(\Gamma)$ belongs to at least one bridge. If $v$ is a vertex of $\Theta_n(\Gamma)$ in the bridge between $\eta_i(x)$ and $\eta_j(x)$, then the set
\[\mathcal{W}(v)=\{\eta_1(x), \ldots , \eta_{n}(x)\}\]
is called the \emph{shadow of $v$}, and $\eta_i(x)$ is the \emph{shadow of $v$ on $\Gamma_i$}, and $x$ is the \emph{shadow of $v$ on $\Gamma$}.
Observe that the $\Theta_n$-extension of $\Gamma$ depends on the vertex $u_0$. Note that if $\Gamma$ is vertex transitive, then the isomorphism type of $\Theta_n(\Gamma)$ is independent of the choice of $u_0$.
\begin{theorem}\label{TcopN}
For any connected graph $\Gamma$ and for any integer $n>0$,
\[ \wcop(\Gamma) \leq \wcop(\Theta_n(\Gamma)) \leq n\cdot\wcop(\Gamma)\] and
\[\scop(\Gamma) \leq \scop(\Theta_n(\Gamma))\leq n\cdot\scop(\Gamma).\]
\end{theorem}
The proof of the theorem is divided into two propositions below which provide the lower and upper bounds for the cop numbers of $\Theta_n(\Gamma)$, respectively.
\begin{proposition}\label{thetaqr}
For any connected graph $\Gamma$ and $n>0$,
\[ \wcop(\Gamma) \leq \wcop(\Theta_n(\Gamma))\] and \[\scop(\Gamma) \leq \scop(\Theta_n(\Gamma)).\]
\end{proposition}
By Theorem~\ref{corocopN1}, to prove Proposition~\ref{thetaqr}, it is enough to show that $\Theta_n(\Gamma)$ quasi-retracts to $\Gamma$. Observe that the functions $\eta_k\colon \Gamma \to \Theta_n(\Gamma)$ preserve distances, that is,
\[ \dist_{\Theta_n(\Gamma)} (\eta_k(x), \eta_k(y)) = \dist_{\Gamma}(x,y) \]
for any pair of vertices $x,y$ of $\Gamma$. In particular they are $(1,0)$-Lipschitz. On the other hand, note that the function $g\colon \Theta_n(\Gamma) \to \Gamma$ that maps any vertex of a bridge between $\eta_i(x)$ and $\eta_j(x)$ to the vertex $x$ is also $(1,0)$-Lipschitz. Since $g\circ \eta_k$ is the identity map on $\Gamma$, it follows that the pair $(\eta_k, g)$ is a $(1,0)$-quasi-retraction of $\Theta_n(\Gamma)$ into $\Gamma$.
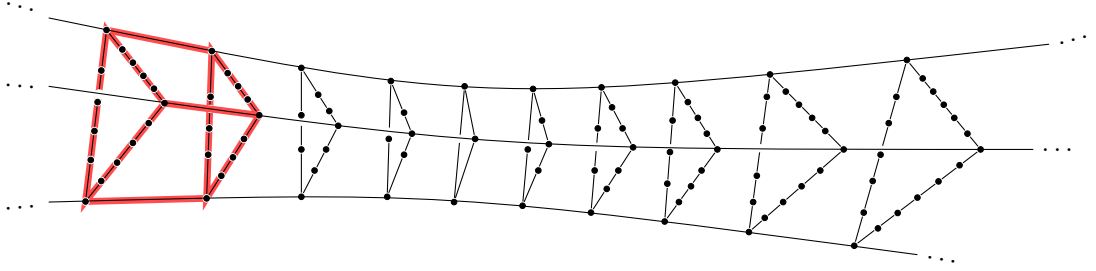
\begin{figure}
	\centering
\begin{tikzpicture}[scale=0.7]
\draw[line width=0.09cm,red!70](-7.3,2.88)--(-8.4,4.27)--(-8.8,1.01)--cycle ;
\draw[line width=0.09cm,red!70](-5.5,2.65)--(-6.4,3.87)--(-6.5,1.07)--cycle ;
\draw[line width=0.09cm,red!70](-8.4,4.27)--(-6.4,3.87);
\draw[line width=0.09cm,red!70](-8.8,1.01)--(-6.5,1.07);
\draw (-1.6,3.2)--(-1.4,2.2)--(-1.8,1)--cycle;
\draw (-0.3,3.15)--(-0.5,0.93)--(0,2.1)--cycle;
\draw (-2.6,2.3)--(-3.07,1.1)--(-3,3.3)--cycle;
\draw (-4,2.45)--(-4.7,3.55)--(-4.7,1.1)--cycle;
\draw (-5.5,2.65)--(-6.4,3.87)--(-6.5,1.07)--cycle;
\draw (-7.3,2.88)--(-8.4,4.27)--(-8.8,1.01)--cycle;
\draw (1.6,2.04)--(1,3.18)--(0.8,0.8)--cycle;
\draw (3.2,2)--(2.4,3.27)--(2.2,0.63)--cycle;
\draw (5.6,2)--(4.2,3.425)--(3.8,0.43)--cycle;
\draw (8.2,2)--(6.8,3.7)--(5.8,0.17)--cycle;
\foreach \Point in {(-1.6,3.2),(-0.3,3.15),(-3,3.3),(-4.7,3.55),(-1.4,2.2), (0,2.1),(-2.6,2.3),(-4,2.45),(-1.8,1),(-0.5,0.93),(-3.07,1.1),(-4.7,1.1),(-5.5,2.65),(-6.4,3.87),(-6.5,1.07),(-7.3,2.88),(-8.4,4.27),(-8.8,1.01),(1.6,2.04),(1,3.18),(0.8,0.8),(3.2,2),(2.4,3.27),(2.2,0.63),(5.6,2),(4.2,3.425),(3.8,0.43),(8.2,2),(6.8,3.7),(5.8,0.17),(6.3,2),(-0.4,2),(-0.2,1.6),(-0.13,2.55),
(-2.75,1.9),(-3.04,2.2),(-2.75,2.7),(-4.23,2),(-4.45,1.6),(-4.17,2.73),(-4.38,3.04),(-4.7,2.65),(-4.7,2),(-6.1,3.45),(-5.9,3.2),(-5.72,2.95),(-5.79,2.2),(-6,1.85),(-6.22,1.5),(-6.425,3),(-6.45,2.4),(-6.47,1.9),(-8.1,3.9),(-7.9,3.65),(-7.7,3.4),(-7.5,3.155),(-7.6,2.5),(-7.9,2.125),(-8.2,1.75),(-8.5,1.4),(-8.7,1.8),(-8.63,2.35),(-8.57,2.9),(-8.5,3.5),(0.93,2.4),(0.87,1.6),(1.32,1.6),(1.1,1.25),(1.4,2.4),(1.2,2.8),(2.3,1.95),(2.35,2.55),(2.25,1.35),(2.47,1),(2.9,1.6),(2.7,1.3),(3,2.3),(2.83,2.6),(2.64,2.9),(4.14,3),(4.06,2.4),(3.95,1.5),(3.88,1),(4.5,3.1),(4.75,2.85),(5,2.6),(5.25,2.35),(5.155,1.6),(4.8,1.3),(4.45,1),(4.1,0.7),(6.3,1.9),(6.15,1.4),(5.98,0.8),(6.455,2.5),(6.6,3),(7.1,3.35),(7.5,2.85),(7.3,3.1),(7.7,2.6),(7.95,2.3),(6.25,0.5),(7,1.085),(6.625,0.7925),(7.8,1.7),(7.4,1.3925),(-1.67,2.2),(-0.4,2.1),(-3.05,2.34),(-4.7,2.53),(-6.43,2.77),(0.9,2.05),(2.3,2.02),(4,2.025)
}{
\filldraw[white] \Point circle (2pt);
}
\filldraw[white] (-8.54,3.045) circle (2.5pt);
\filldraw[white] (-8.55,3.035) circle (2.5pt);
\draw[line width=0.09cm,red!70](-7.3,2.88)--(-5.5,2.65);
\node[rotate = -15] at (-10,4.7) {$\dots$};
\node[rotate = 5] at (-10,0.9) {$\dots$};
\node[rotate = -5] at (-10,3.2){$\dots$};
\node[rotate = 0] at (9.7,2) {$\dots$};
\node[rotate = -10] at (7.5,-0.1) {$\dots$};
\node[rotate = 15] at (10,4.1) {$\dots$};
\foreach \Point in { (-1.6,3.2),(-0.3,3.15),(-3,3.3),(-4.7,3.55),(-6.4,3.87),(-8.4,4.27),(1,3.18),(2.4,3.27),(4.2,3.425),(6.8,3.7), (-1.4,2.2), (0,2.1),(-2.6,2.3),(-4,2.45),(-5.5,2.65),(-7.3,2.88),(1.6,2.04),(3.2,2),(5.6,2),(8.2,2),(-1.8,1),(-0.5,0.93),(-3.07,1.1),(-4.7,1.1),(-6.5,1.07),(-8.8,1.01),(0.8,0.8),(2.2,0.63),(3.8,0.43),(5.8,0.17)
}{
\filldraw \Point circle (1.5pt);
}
\foreach \Point in {(-0.4,2),(-0.2,1.6),(-0.13,2.55),
(-2.75,1.9),(-3.04,2.2),(-2.75,2.7),(-4.23,2),(-4.45,1.6),(-4.17,2.73),(-4.38,3.04),(-4.7,2.65),(-4.7,2),(-6.1,3.45),(-5.9,3.2),(-5.72,2.95),(-5.79,2.2),(-6,1.85),(-6.22,1.5),(-6.425,3),(-6.45,2.4),(-6.47,1.9),(-8.1,3.9),(-7.9,3.65),(-7.7,3.4),(-7.5,3.155),(-7.6,2.5),(-7.9,2.125),(-8.2,1.75),(-8.5,1.4),(-8.7,1.8),(-8.63,2.35),(-8.57,2.9),(-8.5,3.5),(0.93,2.4),(0.87,1.6),(1.32,1.6),(1.1,1.25),(1.4,2.4),(1.2,2.8),(2.3,1.95),(2.35,2.55),(2.25,1.35),(2.47,1),(2.9,1.6),(2.7,1.3),(3,2.3),(2.83,2.6),(2.64,2.9),(4.14,3),(4.06,2.4),(3.95,1.5),(3.88,1),(4.5,3.1),(4.75,2.85),(5,2.6),(5.25,2.35),(5.155,1.6),(4.8,1.3),(4.45,1),(4.1,0.7),(6.3,1.9),(6.15,1.4),(5.98,0.8),(6.455,2.5),(6.6,3),(7.1,3.35),(7.5,2.85),(7.3,3.1),(7.7,2.6),(7.95,2.3),(6.25,0.5),(7,1.085),(6.625,0.7925),(7.8,1.7),(7.4,1.3925)
}{
\filldraw \Point circle (1.5pt);
}
\draw (-9.5,1) .. controls (-2,1.2) .. (7,0);
\draw (-9.5,3.2) .. controls (-1,2) .. (9.2,2);
\draw (-9.5,4.5) .. controls (-1.2,2.8) .. (9.5,4);
		\end{tikzpicture}
	\caption{$\Theta_3$-extension of the infinite path, and a subgraph isomorphic to $\theta_{3,4}$.}
	\label{fig:Theta3Z}
\end{figure}
\begin{proposition}\label{lemTHn}
For any connected graph $\Gamma$ and any integer $n > 0$,
\[ \wcop(\Theta_n(\Gamma)) \leq n\cdot\wcop(\Gamma),\] and
\[\scop(\Theta_n(\Gamma))\leq n\cdot\scop(\Gamma).\]
\end{proposition}
The reader is encouraged to review Section~\ref{subsec:copwin-notation} for notation that is used in the following argument.
\begin{proof}
To prove these inequalities it is enough to show that
if $\Gamma$ is $\mathsf{CopWin}(m,\sigma,\rho,\psi, v, R)$ then $\Theta_n(\Gamma)$ is $\mathsf{CopWin}(mn,\sigma ,\rho ,\psi , \eta_0(v), R )$. Suppose that $\Gamma$ is $\mathsf{CopWin}(m,\sigma,\rho,\psi, v, R)$.
We will describe a winning strategy for the cops $\Theta_n(\Gamma)$ for a game with parameters $(mn,\sigma ,\rho ,\psi , \eta_0(v), R )$ by playing a parallel game on $\Gamma$. The $m\cdot n$ cops on the game on $\Theta_n(\Gamma)$ are indexed by pairs $i,k$ where $0\leq i<m$ and $0\leq k<n$, and their positions after the end of the $j$-stage are denoted by $c_{i,j}^k$. While playing on $\Theta_n(\Gamma)$, we play a parallel game on $\Gamma$ with parameters
$(m,\sigma,\rho,\psi, v, R)$ using a winning strategy for the cops.
The positions of the $m$ cops in $\Gamma$ after the end of the $j$-stage are denoted by $c_{i,j}$. The movements of the cops in $\Gamma$ will determine the moves of the cops in $\Theta_n(\Gamma)$ according to the rule
\[ \eta_k (c_{i,j}) = c^k_{i,j} .\]
The moves of the robber in $\Theta_n(\Gamma)$ will determine the moves of the robber in the game on $\Gamma$ by considering its shadow on $\Gamma$.
It is an observation that a movement of the robber in $\Theta_n(\Gamma)$ determines a valid move of the robber in $\Gamma$. Indeed if $p$ is a path in $\Theta_n(\Gamma)$ such that its shadow in $\Gamma$ has a vertex at a distance less than $\rho$ from a cop in $\Gamma$, then $p$ has a vertex at a distance less than $\rho$ from a cop in $\Theta_n(\Gamma)$.
Note that if the robber on $\Theta_n(\Gamma)$ lies on a bridge that connects $\eta_i(x)$ and $\eta_j(x)$ and there are cops on those two vertices, the robber is trapped between those vertices and the cops can now capture the robber after a finite number of stages. In this situation we will say that the robber is in ``\textit{zugzwang}''.
Observe that if the robber in $\Gamma$ is captured by the end of the $j$-stage, then the robber in $\Theta_n(\Gamma)$ is in zugzwang by the end of the $j$-stage. Hence after a finite number of stages the robber in $\Theta_n(\Gamma) $ is captured.
On the other hand, if at any stage the robber in $\Gamma$ is at a distance larger than $R$ from $v$, then the robber in $\Theta_n(\Gamma)$ is at distance larger than $R$ from $\eta_1(v)$.
The statements of the last two paragraphs show that following the strategy on $\Gamma$ yields a  strategy on $\Theta_n(\Gamma)$ such that if the robber is captured in $\Gamma$ then it is also captured in $\Theta_n(\Gamma)$; and if the robber in $\Gamma$ is pushed away from the $R$-ball centered at $v$, then the robber in $\Theta_n(\Gamma)$ is pushed away from the $R$-ball centered at $\eta_1(v)$.
\end{proof}

\section{Realization of cop numbers}
\begin{theorem}\label{CoroTcop}
Let $\Gamma$ be a connected unbounded graph and $n$ a positive integer.
If $\wcop(\Gamma)=1$ then $\wcop(\Theta_n(\Gamma))=n$.
Analogously, if $\scop(\Gamma)=1$ then $\scop(\Theta_n(\Gamma))=n$.
\end{theorem}
The proof is divided into three parts. First we introduce a finite graph $\theta_{n,m}$ that depends of the integer parameters $m$ and $n$, and we show that on these graphs a robber is never captured if some particular relation between $m,n$ and the parameters of the game holds, see Proposition~\ref{lilTheta}. In the second part, we show that in general the $\Theta_n$-extension $\Theta_n(\Gamma)$ of a graph $\Gamma$ contains induced subgraphs isomorphic to $\theta_{n,m}$; then by almost repeating the argument of the first part, we show that $\Theta_n(\Gamma)$ is robber win for $n-1$ cops subject to a particular relation between $n$ and the parameters of the game, see Proposition~\ref{lilTheta2}. The last part of the section concludes with the proof of the theorem.
\subsection{The finite graph $\theta_{n,m}$.}
Let us consider integers $n\geq1$ and $m\geq 1$. Let
$\theta_{n,m}$ denote the graph defined as follows:
\begin{itemize}
\item Consider the complete graphs $K_n$ and $K_2$, and take the Cartesian product of graphs $\sq{K_n}{K_2}$.
\item Inside the graph $\sq{K_n}{K_2}$ there are two copies of $K_n$, we are going to regard them as two ``\textit{levels}'' on the graph.
\item For every edge between vertices of the first level, subdivide the edge into $m$ edges.
\item For every edge between vertices of the second level, subdivide the edge $m$ into $m+1$ edges.
\end{itemize}
\begin{figure}[ht]
	\centering
	
\begin{tikzpicture}[scale=1]
\foreach \Point in {(-4,1),(4,1.5),(-0.5,2.5),(0.5,0)
}{
\filldraw[blue!80] \Point circle (3pt);
\filldraw \Point circle (2pt);
}
\draw (-4,1)--(-0.5,2.5)--(4,1.5)--(0.5,0)--cycle;
\draw (-0.5,2.5)--(0.5,0);
\filldraw[white] (0,1.25) circle (2.3pt);
\draw (-4,1)--(4,1.5);
\foreach \Point in {(-2.83,1.5),(-1.67,2),(-2.5,0.67),(-1,0.33),(-1.33,1.17),(1.33,1.33),(1.67,0.5),(2.83,1),(0.6,2.25),(2.5,1.83),(-0.17,1.67),(0.17,0.83)
}{
\filldraw \Point circle (2pt);
}
\foreach \Point in {(-0.38,2.2),(-0.92,2.32),(0.38,2.3),(0.935,2.18)
}{
\filldraw[white] \Point circle (1.3pt);
}
\foreach \Point in {(-4,3),(4,3.5),(-0.5,4.5),(0.5,2)
}{
\filldraw[red!80] \Point circle (3pt);
\filldraw \Point circle (2.3pt);
}
\draw (-4,3)--(-0.5,4.5)--(4,3.5)--(0.5,2)--cycle;
\draw (-0.5,4.5)--(0.5,2);
\filldraw[white] (0,3.25) circle (2pt);
\draw (-4,3)--(4,3.5);
\draw (-4,3)--(-4,1);
\draw[dashed] (-0.5,4.5)--(-0.5,2.5);
\draw (4,1.5)--(4,3.5);
\draw (0.5,2) -- (0.5,0);
\foreach \Point in {(0.15,2.9),(0.3,3.27),(-2.25,3.75),(-1.75,2.5),(1.75,4),(2.25,2.75)
}{
\filldraw \Point circle (2pt);
}
		\end{tikzpicture}
	\caption{An illustration of the graph $\theta_{4,2}$}
	\label{fig:t42}
\end{figure}
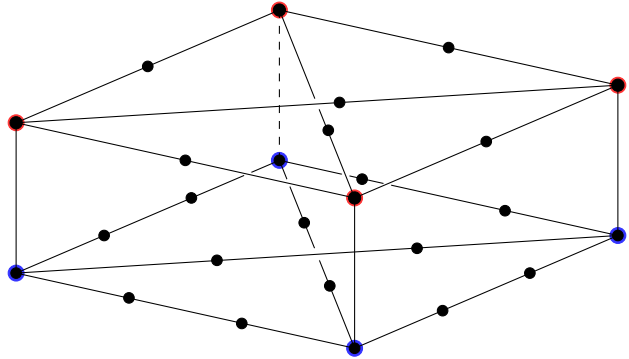
The vertices of $\theta_{n,m}$ corresponding to the vertices of $\sq{K_n}{K_2}$ (before subdivisions) are called \emph{corners}. Observe that the distance between any two corners of $\theta_{n,m}$ at the same level is at least $m$. Two corners of $\theta_{n,m}$ which are adjacent and in different levels are called \emph{sibling corners}. A \emph{bridge} in $\theta_{n,m}$ is any path of length $m$ between corners in the first level, and any path of length $m+1$ between corners in the second level. Hence if $r$ is a corner of $\theta_{n,m}$, then there are $n-1$ bridges with initial vertex $r$.
For example, the red subgraph on Figure~\ref{fig:Theta3Z} is a graph isomorphic to $\theta_{3,4}$, where the first level corresponds to the right triangle, and the second level to the left triangle. Figure~\ref{fig:t42} shows $\theta_{4,2}$, the red and blue vertices are the corners, the red vertices are in the first level, and the blue ones in the second level.

A finite graph $\Gamma$ is said to be
$\mathsf{RobberWin}(n,\sigma,\rho,\psi,\infty)$ if it is not
$\mathsf{CopWin}(n,\sigma,\rho,\psi, R)$ for $R$ equal to  the diameter of $\Gamma$. Roughly speaking,   $\Gamma$ is
$\mathsf{RobberWin}(n,\sigma,\rho,\psi,\infty)$ if a robber with speed $\psi$ has a strategy such that it is never captured by $n$ cops with speed $\sigma$ and reach $\rho$.
\begin{proposition}\label{lilTheta}
Let $n\geq 1$, $m\geq 1$, $\sigma\geq 1$, $\rho>0$, $\psi\geq 1$ be integers. If $\psi> \sigma +2$, $m>2(\sigma+\rho)$ and
\[\rho < \frac m2 - \sigma\ceil*{\frac{m+2}{\psi}} ,\]
then $\theta_{n,m}$ is \emph{$\mathsf{RobberWin}(n-1,\sigma,\rho,\psi, \infty)$}.
\end{proposition}
We introduce some notation and prove a lemma before the proof of the proposition. For each vertex $v$ of $\theta_{n,m}$, choose a pair of sibling corners $E(v)$ such that $ \dist(v, E(v))$ is minimal. Note that $\dist(v, E(v))\leq \frac{m+1}{2}$ for any vertex $v$. Observe that if $S$ is a pair of siblings such that $S\neq E(v)$, then $\dist(v, S)\geq \frac m2$.
Let $S$ and $T$ be pairs of sibling corners and let $v$ be a vertex of $\theta_{n,m}$. Let $\mathcal{B}(S,v)$ denote the set of bridges from a corner in $S$ that contain the vertex $v$; hence $\mathcal{B}(S,v)$ has cardinality $n-1$ if $v\in S$, and cardinality at most one otherwise. Let $\mathcal{B}(S,T)$ denote the set of bridges from a corner in $S$ to a corner of $T$, and let $\mathcal{B}(S)$ the set of all bridges from a corner in $S$. Observe that $\mathcal{B}(S)$ has cardinality $2n-2$.
\begin{lemma}\label{lemma:theta}
Consider the graph $\theta_{n,m}$ and suppose that $m>1$. Let $S$ be a pair of sibling corners, and $C$ a collection of $n-1$ vertices such that $S\cap C = \emptyset$.
Then there is a pair of sibling corners $T$ such that $\dist (T,C)\geq \frac m2$ and either $S =T$ or there is a bridge from a corner in $S$ to a corner in $T$ that has no vertices in $C$.
\end{lemma}
\begin{proof}
For each vertex $v\in C$, let $f(v)=\mathcal{B}(S,E(v))$ if $E(v)\neq S$, and $f(v)=\mathcal{B}(S,v)$ otherwise. Note that $f(v)$ has cardinality at most two. We will consider the set $\bigcup_{v\in C} f(v) \subseteq \mathcal{B}(S)$. We divide the proof into two cases.
\emph{Case 1. $\bigcup_{v\in C} f(v) = \mathcal{B}(S)$.} Let $T=S$ and let us argue that $\dist(T,C)\geq \frac m2$.
Since $C$ has $n-1$ elements, for every $v\in C$ we have that $f(v)$ has cardinality two, hence $S\neq E(v)$, and therefore $\dist(S,v)\geq \frac m2$.
\emph{Case 2. $\bigcup_{v\in C} f(v) \subsetneq \mathcal{B}(S)$.}
Then there is a bridge $p$ in $\mathcal{B}(S)$ that is not in $\bigcup_{v\in C} f(v)$. The initial vertex of $p$ is a corner in $S$ and the end vertex is a corner in some pair of sibling corners $T$. Since $\mathcal{B}(S, v) \subset \mathcal{B}(S, E(v))$ if $S\neq E(v)$, we have that $\mathcal{B}(S,v) \subseteq f(v)$ for every $v\in C$. Since $p$ is not in
$\bigcup_{v\in C} f(v)$, it follows that no vertex of $p$ is in $C$. It is left to prove that $\dist(T,C)\geq \frac m2$. Let $v \in C$. In the case that $E(v)\neq S$, we have that $p \not \in f(v) = \mathcal{B}(S, E(v))$. It follows that $T\neq E(v)$ and hence $\dist(T,v)\geq \frac m 2$. On the other hand, if $E(v) = S$ then $\dist(S, v)\leq \frac{m+1}{2}$ and hence \[\dist(T,v)\geq \dist(S,T)-\dist (S,v)\geq m-\frac{m+1}{2} =\frac{m+1}{2}\geq \frac m 2.\qedhere.\]
\end{proof}
\begin{proof}[Proof of Proposition~\ref{lilTheta}]
We need to show that the robber has a strategy on which it is never captured. Fix the parameters $m$, $n$, $\sigma$, $\rho$ and $\psi$ satisfying the hypotheses.
The strategy for the robber on $\theta_{n,m}$ is described as follows. Let $c_{i,t}$ be the position of the $i$-th cop after his move during the $t$-stage of the game, let $r_t$ denote the position of the robber at the end of the $t$-stage, and let $C_t$ denote the set of vertices occupied by the cops at the end of the $t$-stage.
The initial position $r_0$ is chosen as follows. Let $C_0$ be the set of initial positions of the cops, this is a set with $n-1$ vertices. By Lemma~\ref{lemma:theta}, there is a pair of sibling corners $T$ such that $\dist( C_0, T)\geq \frac m2> \rho+\sigma$. Let the initial position of the robber be a corner $r_0$ in $T$. Observe that this is a safe position for the robber, and that the robber will not be captured during the first move of the cops, that is,
\[\dist(C_1, r_0)\geq \frac m2 -\sigma >\rho.\]
Suppose that during the $(t+1)$-stage, after the cops have moved, the robber is located in a \emph{corner} $r_t$ and the robber has not been captured. In particular,
\[\dist(C_{t+1}, r_t)> \rho.\]
If $\dist(C_{t+1}, r_t)\geq \frac{m}{2} >\rho+\sigma$ then the robber does not move in this stage, that is, $r_{t+1}=r_t$. In particular, the robber is still located in a corner and the robber will not be captured after the next move of the cops, that is,
\[\dist(C_{t+2}, r_{t+1})>\rho.\]
Now we consider the case that $\dist(C_{t+1}, r_t)< \frac{m}{2}$. Let $S$ be the pair of sibling corners containing $r_t$. By Lemma~\ref{lemma:theta}, there is a pair of sibling corners $T$ such that $\dist(T,C_{t+1})\geq \frac m2$ and a bridge $p$ from a corner in $S$ to a corner of $T$ such that $p$ has no vertices in $C_{t+1}$.
For the following stages of the game, the robber will move along the bridge $p$ until it reaches a corner in $T$.
The robber will reach that corner in $T$ at the end of the $t_*$-stage, where
\[t_* = t + \ceil*{\frac{m+2}{\psi}} ,\]
since $p$ has length at most $m+1$ and $r_t$ is at a distance at most one from the initial vertex of $p$. At the end of the $t_*$-stage the robber will be located at a corner. We will show that the robber is not captured during the $s$-stage for every $s\in\{t+1,\ldots ,t_*\}$, and hence this will complete the description of a winning strategy for the robber.
Let $s\in\{t+1,\ldots,t_*\}$.
Let $\tilde r_t$ and $r_{t_*}$ denote the initial and terminal vertices of the bridge $p$, respectively. We need to argue that the robber will not be captured while moving along the bridge towards $r_{t_*}$, that is, we need to show that
\[\dist(c_{i,s+1},r_s)>\rho\]
for every $i\in \{1,\ldots n\}$. Fix an $i$, and to simplify notation, let $c_{s+1}$ and $c_{t+1}$ denote $c_{i,s+1}$ and $c_{i,t+1}$, respectively.
Consider a geodesic $q$ from $c_{t+1}$ to $r_s$.
Since $c_{t+1}$ is not a vertex in the bridge $p$, the geodesic $q$ has to pass through the initial $\tilde r_t$ or terminal $r_{t_*}$ vertex of the bridge $p$. The rest of the argument is split into two cases.
\emph{Case 1.} Suppose that $q$ passes through $r_{t_*}$. Then the length of $q$ is at least $\frac{m}{2}$, since $r_{t_*}\in T$ and $\dist(c_{t+1}, r_{t_*})\geq \dist(C_{t+1}, T)\geq\frac m2$. Therefore
\[\begin{split}
\dist (c_{s+1}, r_s) & \geq \dist (c_{t+1}, r_s) - \dist (c_{t+1}, c_{s+1})\\ & \geq \dist (c_{t+1}, r_s) - (s-t)\sigma \\ & \geq \frac m2 - \sigma\ceil*{\frac{m+2}{\psi}} \\
& >\rho.
\end{split}\]
\emph{Case 2.} The geodesic $q$ from $c_{t+1}$ to $r_s$ passes through $\tilde r_{t}$. In this case
\[ \rho+(s-t)\psi-2 \leq \dist(c_{t+1}, \tilde r_t)+\dist(\tilde r_t, r_s) =\dist (c_{t+1},r_s). \]
Suppose, by contradiction, that $\dist(c_{s+1},r_s)\leq \rho$. Then
\[\begin{split}
\rho+(s-t)\psi-2 & \leq \dist (c_{t+1},r_s) \\
& \leq \dist(c_{t+1},c_{s+1})+\dist(c_{s+1}, r_s)\\
& \leq (s-t)\sigma + \rho.
\end{split}\]
This implies that $(s-t)(\psi-\sigma)\leq 2$. Since $1\leq s-t$, we have that $\psi-\sigma\leq 2$ which contradicts our assumption. Therefore, $\dist(c_{s+1},r_s)> \rho$.
\end{proof}
\subsection{Copies of $\theta_{n,m}$ in $\Theta_n(\Gamma)$}
In this part, we describe some canonical subgraphs of $\Theta_n(\Gamma)$ isomorphic to $\theta_{n,m}$, see Figure ~\ref{fig:Theta3Z} for an illustration. These subgraphs allow us to prove the following result by mimicking the argument proving Proposition~\ref{lilTheta}.
\begin{proposition}\label{lilTheta2}
Let $\Gamma$ be a connected graph, let $n$ be a positive integer, let $u_0$ be a vertex of $\Gamma$ and consider the $\Theta_n$-extension $\Theta_n(\Gamma)$ centered at $u_0$.
Let $m\geq 1$, $\sigma\geq 1$, $\rho>0$, $\psi\geq 1$ be integers. Suppose $\psi> \sigma +2$, $m>2(\sigma+\rho)$ and
\[\rho < \frac m2 - \sigma\ceil*{\frac{m+2}{\psi}} .\]
If there are vertices $x,y$ of $\Gamma$ such that
\[ \dist_\Gamma(u_0,x)=m,\qquad \dist_\Gamma(u_0,y)=m+1,\qquad \dist_\Gamma(x,y)=1, \]
then $\Theta_n(\Gamma)$ is $\mathsf{RobberWin}(n-1,\sigma,\rho,\psi, 2m+2)$.
\end{proposition}
Below we use the notation introduced in Section~\ref{sec:ThetaExtensions} for the definition of $\Theta$-extensions of graphs. From here on we work under the assumptions of Proposition~\ref{lilTheta2}.
Let $\theta_n(\Gamma,x,y)$ be the induced subgraph of $\Theta_n(\Gamma)$ with vertex set
\[ V(\theta_n(\Gamma,x,y)) = \bigcup_{0\leq i<j<N}\left(V_{i,j,x}\cup V_{i,j,y}\right).\]
Observe that there is a natural isomorphism of graphs between the subgraph $\theta_n(\Gamma,x,y)$ and the graph $\theta_{n,m}$. Specifically, the set of vertices $\bigcup_{1\leq i<j\leq n}V_{i,j,x}$ corresponds to the first level, the set $\bigcup_{1\leq i<j\leq n}V_{i,j,y}$ correspond to the second level, and for each $i\in\{1,\ldots ,n\}$ the vertices $\eta_i(x)$ and $\eta_i(y)$ are sibling corners.
Observe that that $\theta_n(\Gamma,x,y)$ is a convex subgraph of $\Theta_n(\Gamma)$ in the sense that any geodesic path in $\Theta_n(\Gamma)$ with endpoints in $\theta_n(\Gamma,x,y)$ is contained in $\Theta_{n,m}$.
Abusing notation, from here on, {\bf let $\theta_{n,m}$ denote the subgraph $\theta_n(\Gamma,x,y)$ of $\Theta_n(\Gamma)$}.
\begin{lemma}\label{lemma:theta2}
Suppose that $m>1$. Let $S$ be a pair of sibling corners of $\theta_{n,m}$, and $C$ a collection of $n-1$ vertices of $\Theta_n(\Gamma)$ such that $S\cap C = \emptyset$.
Then there is a pair of sibling corners $T$ of $\theta_{n,m}$ such that $\dist_\Gamma (T,C)\geq \frac m2$ and either $S =T$ or there is a bridge from a corner in $S$ to a corner in $T$ that has no vertices in $C$.
\end{lemma}
\begin{proof}
For each vertex $v$ of $\Theta_n(\Gamma)$, let $E(v)$ be the empty set if $\dist_\Gamma(v, \theta_{n,m})>\frac{m+1}{2}$.
If $\dist_\Gamma (v, \theta_{n,m})\leq \frac{m+1}{2}$ then choose a pair of sibling corners $E(v)$ such that $ \dist_\Gamma(v, E(v))$ is minimal. Observe that
\begin{itemize}
\item if $E(v) $ is nonempty, then $\dist_\Gamma(v, E(v))\leq \frac{m+1}{2}$.
\item if $U$ is a pair of sibling corners of $\theta_{n,m}$ with $U\neq E(v)$, then $\dist_\Gamma(v, U)\geq \frac m2$.
\end{itemize}
The second observation is a consequence of $\theta_{n,m}$ being a convex subgraph of $\Theta_n(\Gamma) $, and that different pairs of sibling corners of $\theta_{n,m}$ are at distance at least $m$.
The proof concludes by literally following the proof of Lemma~\ref{lemma:theta2} with the extended definition of $E(v)$ above, the convention that if $v\in C$ and $E(v)=\emptyset$ then $f(v)=\emptyset$, and using distances from $\Theta_n(\Gamma)$.
\end{proof}
\begin{proof}[Proof of Proposition~\ref{lilTheta2}]
Let $R=2m+2$. A game with $n-1$ cops with speed $\sigma$ and reach $\rho$ cannot protect the $R$-ball
$B_R(\eta_1(u_0))$ in $\Theta_n(\Gamma)$ from a robber with speed $\psi$. Indeed, there is a  strategy for the robber such that he always moves inside the subgraph $\theta_n(\Gamma, x,y)=\theta_{n,m}$ and avoids capture during the game. Since the ball $B_R(\eta_1(u_0))$ contains the subgraph $\theta_{n,m}$ this yields a winning strategy for the robber. The strategy is the one described in the proof of Proposition~\ref{lilTheta} and the proof that the strategy is successful for the robber is the same modulo using Lemma~\ref{lemma:theta2} instead of Lemma~\ref{lemma:theta}.
\end{proof}
\subsection{Proof of Theorem~\ref{CoroTcop}}
\begin{proof}[Proof of Theorem~\ref{CoroTcop}]
Suppose that $\Theta_n(\Gamma)$ is centered at the vertex $u_0$ of $\Gamma$. Let $\sigma$ and $\rho$ be arbitrary positive integers, let $\psi=2\sigma+3$. Let $m$ be a positive integer such that
\[ m > 2(\sigma+\rho) \text{ and } \rho < \frac m2 - \sigma\ceil*{\frac{m+2}{\psi}}.\]
Note that such $m$ exists since $\frac \sigma \psi < \frac 12$.
Since $\Gamma$ is an unbounded locally finite connected graph, it contains vertices $x$ and $y$ such that \[ \dist_\Gamma(u_0,x)=m,\qquad \dist_\Gamma(u_0,y)=m+1,\qquad \dist_\Gamma(x,y)=1. \]
Then Proposition~\ref{lilTheta2} implies that $\Theta_n(\Gamma)$ is $\mathsf{RobberWin}(n-1,\sigma,\rho,\psi, 2m+2)$. Since $\sigma$ and $\rho$ were arbitrary, and the choice of $\psi$ is independent of $\rho$, it follows that
\[ n-1 < \scop(\Theta_n(\Gamma)) \leq \wcop(\Theta_n(\Gamma)).\]
By Theorem~\ref{TcopN},
\[ \wcop(\Theta_n(\Gamma)) \leq n\cdot \wcop(\Gamma)\qquad \text{and}\qquad \scop(\Theta_n(\Gamma)) \leq n\cdot \scop(\Gamma).\]
Therefore if $\wcop(\Gamma)=1$ then $ \wcop(\Theta_n(\Gamma)) =n$, and analogously if
$\scop(\Gamma)=1$ then $ \scop(\Theta_n(\Gamma)) =n$.
\end{proof}
\bibliographystyle{alpha} 
\bibliography{ref}

\begin{thebibliography}{DMPT17}

\bibitem[Alo94]{Alonso}
J.~M. Alonso.
\newblock Finiteness conditions on groups and quasi-isometries.
\newblock {\em Journal of Pure and Applied Algebra}, 95, 1994.

\bibitem[BC09]{article2}
Anthony Bonato and Ehsan Chiniforooshan.
\newblock Pursuit and evasion from a distance: algorithms and bounds.
\newblock In {\em A{NALCO}09---{W}orkshop on {A}nalytic {A}lgorithmics and
  {C}ombinatorics}, pages 1--10. SIAM, Philadelphia, PA, 2009.

\bibitem[BCP10]{article3}
Anthony Bonato, Ehsan Chiniforooshan, and Paweł Prałat.
\newblock Cops and robbers from a distance.
\newblock {\em Theoretical Computer Science}, 411(43):3834--3844, 2010.

\bibitem[BH99]{brid}
Martin~R. Bridson and Andr\'{e} Haefliger.
\newblock {\em Metric spaces of non-positive curvature}.
\newblock Springer-Verlag, New York, 1999.

\bibitem[BN11]{Bonato2011}
Anthony Bonato and Richard~J. Nowakowski.
\newblock {\em The game of cops and robbers on graphs}, volume~61 of {\em
  Student Mathematical Library}.
\newblock American Mathematical Society, Providence, RI, 2011.

\bibitem[CCNV10]{article1}
Jérémie Chalopin, Victor Chepoi, N.~Nisse, and Yann Vaxès.
\newblock Cop and robber games when the robber can hide and ride.
\newblock {\em SIAM Journal on Discrete Mathematics}, 25, 01 2010.

\bibitem[CCPP14]{chapolin}
J\'{e}r\'{e}mie Chalopin, Victor Chepoi, Panos Papasoglu, and Timoth\'{e}e
  Pecatte.
\newblock Cop and robber game and hyperbolicity.
\newblock {\em SIAM J. Discrete Math.}, 28(4):1987--2007, 2014.

\bibitem[DGP11]{DGP11}
Fran\c{c}ois Dahmani, Vincent Guirardel, and Piotr Przytycki.
\newblock Random groups do not split.
\newblock {\em Math. Ann.}, 349(3):657--673, 2011.

\bibitem[DK18]{DK18}
Cornelia Dru\c{t}u and Michael Kapovich.
\newblock {\em Geometric group theory}, volume~63 of {\em American Mathematical
  Society Colloquium Publications}.
\newblock American Mathematical Society, Providence, RI, 2018.
\newblock With an appendix by Bogdan Nica.

\bibitem[DL01]{DiLe01}
Reinhard Diestel and Imre Leader.
\newblock A conjecture concerning a limit of non-{C}ayley graphs.
\newblock {\em J. Algebraic Combin.}, 14(1):17--25, 2001.

\bibitem[DMPT17]{firefighter}
Danny Dyer, Eduardo Mart\'{\i}nez-Pedroza, and Brandon Thorne.
\newblock The coarse geometry of {H}artnell's firefighter problem on infinite
  graphs.
\newblock {\em Discrete Math.}, 340(5):935--950, 2017.

\bibitem[EFW12]{EsFiWh12}
Alex Eskin, David Fisher, and Kevin Whyte.
\newblock Coarse differentiation of quasi-isometries {I}: {S}paces not
  quasi-isometric to {C}ayley graphs.
\newblock {\em Ann. of Math. (2)}, 176(1):221--260, 2012.

\bibitem[Gro87]{gromov}
Mikhael Gromov.
\newblock Hyperbolic groups.
\newblock In {\em Essays in group theory}, pages 75--263. Springer, 1987.

\bibitem[Gro93]{Gromov93}
M.~Gromov.
\newblock Asymptotic invariants of infinite groups.
\newblock In {\em Geometric group theory, {V}ol. 2 ({S}ussex, 1991)}, volume
  182 of {\em London Math. Soc. Lecture Note Ser.}, pages 1--295. Cambridge
  Univ. Press, Cambridge, 1993.

\bibitem[Har95]{Hartnell}
B.~L. Hartnell.
\newblock Firefighter! an application of domination.
\newblock 25th Manitoba Conference on Combinatorial Mathematics and Computing,
  University of Manitoba in Winnipeg, Canada, 1995.

\bibitem[Hat22]{HatcherTopNumbers}
Allen Hatcher.
\newblock Topology of numbers.
\newblock \url{https://pi.math.cornell.edu/~hatcher/TN/TNbook.pdf}, 2022.
\newblock Accessed: 2022-07-07.

\bibitem[MPP23]{MaPr23}
Eduardo Mart\'{\i}nez-Pedroza and Tomasz Prytu{\l}a.
\newblock Coarse geometry of the fire retaining property and group splittings.
\newblock {\em Geom. Dedicata}, 217(2):Paper No. 40, 18, 2023.

\bibitem[NW83]{Nowa}
Richard Nowakowski and Peter Winkler.
\newblock Vertex-to-vertex pursuit in a graph.
\newblock {\em Discrete Mathematics}, 43(2):235--239, 1983.

\bibitem[Qui78]{Quilliot}
Alain Quilliot.
\newblock {\em Jeux positionnels et Propri{\'e}t{\'e} de Helly}.
\newblock PhD thesis, Th{\`e}se de 3{\`e}me Cycle, 1978.

\end{thebibliography}
\end{document}